\documentclass[11pt]{amsart}
\usepackage{etex}
\usepackage{a4wide,amsfonts,graphicx,
amssymb,stmaryrd,amscd,amsmath,latexsym,amsbsy,amsthm}

\usepackage{pdflscape,geometry}

\usepackage{longtable}

\usepackage[all,knot,poly]{xy}
\usepackage[all,knot,poly]{xy}
\usepackage{color}
\usepackage{ifpdf}
\ifpdf 
  \usepackage{graphicx}
  \DeclareGraphicsExtensions{.pdf,.png,.mps}
  \usepackage{pgf}
  \usepackage{tikz}
  \usetikzlibrary{matrix,arrows,decorations.pathmorphing}
   \usepackage{tikz-cd}
\else 
  \usepackage{graphicx}
  \DeclareGraphicsExtensions{.eps,.bmp}
  \usepackage{pgf}
  \usepackage{tikz}
  \usetikzlibrary{matrix,arrows,decorations.pathmorphing}
  \usepackage{tikz-cd}
  \usepackage{pstricks}
\fi
\usepackage{epic,bez123}
\usepackage{floatflt}
\usepackage{wrapfig}
\usepackage[colorinlistoftodos]{todonotes}
\usepackage{hyperref}
\hypersetup{
  colorlinks   = true,    
  urlcolor     = blue,    
  linkcolor    = blue,    
  citecolor    = red      
}
\newtheorem{thm}{Theorem}[section]

\newtheorem{cor}[thm]{Corollary}
\newtheorem{lem}[thm]{Lemma}
\newtheorem{prop}[thm]{Proposition}
\newtheorem{defn}[thm]{Definition}

\newtheorem{rem}[thm]{Remark}

\newcommand{\mc}{\mathcal}

\newcommand{\bs}{\mathbf {s}}

\newcommand{\ep}{\epsilon}

\newcommand{\Ac}{\mathcal{A}}
\newcommand{\Bc}{\mathcal{B}}

\newcommand{\Cb}{\mathbb{C}}

\newcommand{\Hc}{\mathcal{H}}
\newcommand{\Hb}{\mathbb{H}}
\newcommand{\Fc}{\mathcal{F}}
\newcommand{\Ic}{\mathcal{I}}
\newcommand{\Qc}{\mathcal{Q}}
\newcommand{\Oc}{\mathcal{O}}
\newcommand{\Qb}{\mathbb{Q}}
\newcommand{\Rc}{\mathcal{R}}
\newcommand{\Kf}{\mathfrak{K}}
\newcommand{\Lf}{\mathfrak{L}}
\newcommand{\Rf}{\mathfrak{R}}
\newcommand{\Rfu}{{\mathfrak{R}(\mathfrak{U})}}
\newcommand{\Rfut}{{\tilde{\mathfrak{R}}(\mathfrak{U})}}
\newcommand{\Rfust}{{\mathfrak{R}(\mathfrak{U}^*)}}
\newcommand{\Rfutst}{{\mathfrak{R}(\tilde{\mathfrak{U}}^*)}}

\newcommand{\Sc}{\mathcal{S}}
\newcommand{\Cf}{\mathfrak{C}}
\newcommand{\Xf}{\mathfrak{X}}
\newcommand{\Uc}{\mathcal{U}}
\newcommand{\Uf}{\mathfrak{U}}
\newcommand{\Uft}{\tilde{\mathfrak{U}}}
\newcommand{\Ufst}{{\mathfrak{U}^*}}
\newcommand{\Uftst}{{\tilde{\mathfrak{U}}^*}}

\newcommand{\om}{\omega}

\newcommand{\vb}{\mathbf{v}}

\newcommand{\Xc}{\mathcal{X}}
\newcommand{\Yc}{\mathcal{Y}}

\newcommand{\Zb}{\mathbb{Z}}
\newcommand{\Zc}{\mathcal{Z}}

\begin{document}

\title[A uniform classification of the discrete series]{A uniform classification of discrete series representations of affine Hecke algebras}

\author{Dan Ciubotaru}
        \address[D. Ciubotaru]{Mathematical Institute\\ University of
          Oxford\\ Andrew Wiles Building\\ Oxford, OX2 6GG, UK}
        \email{dan.ciubotaru@maths.ox.ac.uk}

\author{Eric Opdam}
\address[E. Opdam]{Korteweg-de Vries Institute for Mathematics\\Universiteit van Amsterdam\\Science Park 904\\ 1098 XH Amsterdam, The Netherlands}
\email{e.m.opdam@uva.nl}

\date{\today}
\thanks{This research was supported  by 
ERC-advanced grant no. 268105. It is a pleasure to thank 
Xuhua He and Maarten Solleveld for useful discussions and 
comments.}

\begin{abstract} We give a new and independent parameterization of 
the set of discrete series characters of an affine Hecke algebra 
$\mathcal{H}_{\vb}$, in terms of a canonically defined basis $\Bc_{gm}$ of a certain 
lattice of virtual elliptic characters of the underlying (extended) affine Weyl group. 
This classification applies to all semisimple affine Hecke algebras 
$\mathcal{H}$, and to all $\vb\in\Qc$, where $\Qc$ denotes the vector group of positive real 
(possibly unequal) Hecke parameters for $\mathcal{H}$. 
By analytic Dirac induction we define for each $b\in \Bc_{gm}$ a continuous (in the sense of \cite{OpdSol2}) family 
$\Qc^{reg}_b:=\Qc_b\backslash\Qc_b^{sing}\ni\vb\to\textup{Ind}_{D}(b;\vb)$, such that  
$\epsilon(b;\vb)\textup{Ind}_{D}(b;\vb)$ (for some $\epsilon(b;\vb)\in\{\pm 1\}$)
is an irreducible discrete series character of $\mathcal{H}_{\vb}$.
Here $\Qc^{sing}_b\subset\Qc$ is a finite union of hyperplanes in $\Qc$. 
 
In the non-simply laced cases we show that the families of virtual discrete series characters 
$\textup{Ind}_{D}(b;\vb)$ are piecewise rational in the parameters $\vb$.
Remarkably, the formal degree of $\textup{Ind}_{D}(b;\vb)$ in such piecewise rational 
family turns out to be rational. This implies that for each $b\in \Bc_{gm}$ there 
exists a universal rational constant $d_b$ determining  
the formal degree in the family of discrete series characters 
$\epsilon(b;\vb)\textup{Ind}_{D}(b;\vb)$. 
We will compute the canonical constants $d_b$, and the signs $\epsilon(b;\vb)$.
For certain geometric parameters we will provide the comparison with 
the Kazhdan-Lusztig-Langlands classification.
\end{abstract}
\maketitle
\tableofcontents
\section{Motivation and goals}

Let $\mathcal{R}=(X,R_0,Y,R_0^\vee,F_0)$ be a based root datum, with (extended) affine Weyl 
group $W:=W_0\ltimes X$, where $W_0=W(R_0)$ is the finite Weyl group associated with the 
root system $R_0$. Let $R_{0,+}$ denote the positive roots. We denote by $S_0$ the set of simple reflections of $W_0$. Associated to $\Rc$ 
one has a canonical Laurent polynomial algebra $\Lambda$ generated by invertible 
"Hecke parameters", 
and the generic extended affine Hecke algebra 
$\Hc_\Lambda$  over $\Lambda$ (see e.g. \cite{OpdSol2}). Let $\Qc$ be the real vector group of the 
algebraic torus associated with 
$\Lambda$. If $\vb\in\Qc$ then we denote by $\Hc_\vb$ the corresponding specialization of $\Hc_\Lambda$.
\subsubsection{Uniform classification of the discrete series}
Let $\overline{\Rc}_\Zb(W)$ be the lattice of elliptic virtual characters of $W$, equipped with the Euler-Poincar\' e
pairing, and let $\overline{\Rc}_\Zb(\Hc_\vb)$ denote the lattice of elliptic virtual characters of $\Hc_\vb$. If $\pi$ is an element of $\Rc_\Zb(\Hc_\vb)$, let us denote by $\overline\pi$ its image in $\overline{\Rc}_\Zb(\Hc_\vb)$. 

In this paper we will use a basic tool, the so-called "scaling map": 
\begin{align}\label{eq:lim}
\lim_{\vb\to 1}:~ \overline{\Rc}_\Zb(\Hc_{\vb})&\to \overline{\Rc}_\Zb(W)\\
\nonumber [\pi]&\to [\lim_{\ep\to 0}\pi_{\vb^\epsilon}]
\end{align}
where $\pi_{\vb^\epsilon}:=\pi\circ  j_\epsilon^{-1}$.
Here $j_\epsilon:\Hc^{an}_\vb(U)\to  \Hc^{an}_{\vb^\ep}(\sigma_\ep(U))$ ($\ep>0$) is the 
isomorphism between the analytic localizations $\Hc^{an}_\vb(U)$ and $\Hc^{an}_{\vb^\ep}(\sigma_\ep(U))$ 
of the affine Hecke algebra as introduced in \cite[Theorem 5.3]{Opd1}. 
It is easy to see that the family of isomorphisms 
$\{j_\epsilon^{-1}\}_{\epsilon>0}$ 
has a well defined limit at $\epsilon=0$, defining a homomorphism $i_0:\mathbb{C}[W]\to\Hc^{an}_\vb(U)$ 
(cf. \cite[Prop. 4.1.2]{Sol2}). 
This explains the existence of the desired "scaling map"  
$\lim_{\vb\to 1}:~ \overline{\Rc}_\Zb(\Hc_{\vb})\to \overline{\Rc}_\Zb(W)$ as in (\ref{eq:lim}). 
The isomorphisms $j_\epsilon$ ($\epsilon>0$) induce isometric isomorphisms \cite[Theorem 3.5(b)]{OpdSol}: 
\begin{equation}
(j_\epsilon^{-1})^*: \overline{\Rc}_\Zb(\Hc_{\vb})\to \overline{\Rc}_\Zb(\Hc_{\vb^\ep})
\end{equation}
Consequently, the limit $\lim_{\vb\to 1}$ of (\ref{eq:lim}) is an isometry too   
\cite[Theorem 3.5(b)]{OpdSol}.

Let us denote by $\Yc_{\vb}\subset \overline{\Rc}_\Zb(W)$ the image of this map, and 
by $\Yc_{\vb-m}\subset \Yc_{\vb}$ the image of the sublattice of $\overline{\Rc}_\Zb(\Hc_\vb)$ 
of the virtual discrete series characters of $\Hc_\vb$. 
Let $\Yc_{gm}\subset \overline{\Rc}_\Zb(W)$ 
be the smallest sublattice which contains all lattices $\Yc_{\vb-m}$ (see Definition \ref{prop:Ygm}). 
We call $\Yc_{gm}$ the lattice of 
\emph{generically massive} elliptic characters of $W$, and $\Yc_{\vb-m}$ the sublattice of 
$\vb$-massive elliptic characters of $W$.

The lattice $\Yc_{gm}$ possesses 
a distinguished orthonormal basis $\Bc_{gm}$ characterized by a positivity property to be explained below.
To each 
$b\in \Bc_{gm}$ we will assign a subset (nonempty by definition) $\Qc_b^{reg}\subset \Qc$ of the space of parameters by:  
\begin{equation*}
\Qc_b^{reg}:=\{\vb\in\Qc\mid \exists\text{\ an\ irreducible\ discrete\ series\ }\pi\text{\ of\ } \Hc_\vb\text{\ such\ that\ }
\lim_{\ep\to 0}\overline{\pi_{\vb^\ep}}=b\}
\end{equation*}
According to \cite{OpdSol2},  the complement $\Qc^{sing}_b$ of $\Qc^{reg}_b$ is a union of finitely many 
hyperplanes (depending on $b$).  
For $y\in\Yc_{gm}$ in general we put 
$\Qc_y^{reg}:=\cap_{b\in\textup{Supp}(y)}\Qc_b^{reg}$. 
Combining with the technique of analytic Dirac induction (introduced in \cite{COT} in the context of graded affine Hecke algebras)
we can associate a family of virtual discrete 
series characters $\Qc_y^{reg}\ni \vb\to 
\textup{Ind}_{D}(y;\vb)\in\Rc_\Zb(\Hc_\vb)$ (called the Dirac induction of $y$ at $\vb$), which depends 
linearly on $y$, for each fixed $\vb\in \Qc_y^{reg}$. If $b\in \Bc_{gm}$ and $\vb\in \Qc_b^{reg}$ then 
$\textup{Ind}_{D}(b;\vb)$ is an irreducible character up to a sign, characterized by the property 
$\lim_{\ep\to 0}(\overline{\textup{Ind}_{D}(y;\vb^\ep)})=y$.
We will prove (see Proposition \ref{prop:limlim}) that the family 
$\textup{Ind}_{D}(y;\vb)$ depends 
continuously on $\vb\in\Qc_y^{reg}$ in the sense of \cite{OpdSol2}. 
The ''Vogan conjecture'' (cf. \cite{COT}) allows to  compute the generic central character $W_0r_b$ of 
$\textup{Ind}_{D}(b)$ explicitly, where $r_b\in T_\Lambda:=\textup{Hom}(X,\Lambda^\times)$ is a 
\emph{generic 
residual point}  
in the sense of \cite{OpdSol2}. 

To such an orbit of generic residual points $W_0r_b$ we associated in \cite{OpdSol2} an explicit rational function 
$m_b:=m_{W_0r_b}$ on $\Qc$ which is regular on $\Qc$, and with the property that 
$\Qc_b^{reg}=\{\vb\in\Qc\mid m_b(\vb)\not=0\}$.
\begin{thm}\label{thm:uniform} Extend the notion of formal degree $\textup{fdeg}$ 
linearly to the Grothendieck group of finite dimensional tempered representations $\Rc_{\Zb,temp}(\Hc_\vb)$. 
Observe that in this way, the function $\textup{fdeg}$ naturally descends to 
$\overline{\Rc}_{\Zb,temp}(\Hc_\vb)=\overline{\Rc}_{\Zb}(\Hc_\vb)$.
\begin{itemize}
\item[(a)] $\Yc_{gm}$ has a \emph{unique orthonormal} basis 
$\Bc_{gm}$ such that for all $b\in \Bc_{gm}$, and for all $\vb\in\Qc_b^{reg}$, 
$d_b(\vb):=m_b(\vb)^{-1}\textup{fdeg}(\overline{\textup{Ind}_{D}(b;\vb)})>0$.
\item[(b)] 
$\overline{\textup{Ind}_{D}(b;\vb)}$ is represented by a virtual character $\textup{Ind}_{D}(b;\vb)$ of $\Hc_\vb$ which is 
plus or minus an irreducible discrete series. 
\item[(c)] The central character of $\textup{Ind}_{D}(b;\vb)$ is the specialization at $\vb$ of a $W_0$-orbit of generic residual points 
$W_0r_b$, with
$r_b\in T_\Lambda:=\textup{Hom}(X,\Lambda^\times)$.
\item[(d)]
The family $\textup{Ind}_{D}(b;\vb)$ depends 
continuously on $\vb\in\Qc_b^{reg}$ in the sense of \cite{OpdSol2}. 
 \item[(e)] For all $b\in\Bc_{gm}$, 
the signature function $\Qc_b^{reg}\ni\vb\to\epsilon(b;\vb)\in\{\pm 1\}$ such that $\epsilon(b;\vb)\textup{Ind}_{D}(b;\vb)$ 
is an irreducible discrete series character, is locally constant. 
\item[(f)] Define, for each $\vb\in \Qc$, $\Bc_{\vb-m}=\{b\in\Bc_{gm}\mid m_b(\vb)\not=0\}$. 
The assignment 
\begin{equation*}
\Bc_{\vb-m}\ni b\to \epsilon(b;\vb)\textup{Ind}_{D}(b;\vb)
\end{equation*}
yields a canonical bijection between $\Bc_{\vb-m}$ and the set of irreducible discrete series characters of $\Hc_\vb$.
\item[(g)] For all $b\in \Bc_{gm}$: 
$d_b(\vb)=d_b$ is \emph{independent} of $\vb\in \Qc_b^{reg}$ 
(where $d_b(\vb)$ is the 
positive function defined in (a)), and 
$d_b\in\Qb_+$. In other words, for all $\vb\in\Qc_b^{reg}$:
\begin{equation*}
\textup{fdeg}(\textup{Ind}_{D}(b;\vb))=d_bm_b(\vb)
\end{equation*}
(a rational function of $\vb$, regular in all points of $\Qc$, with zero locus $\Qc^{sing}_b$).
\end{itemize}
\end{thm}
This represents first of all a new classification of the 
discrete series of $\Hc_\vb$, which is uniform in the sense that it applies to all irreducible root data 
and all parameters $\vb\in \Qc$. It is explicit in the sense that for cases 
where a classification of discrete series has been given in other terms in the literature 
(e.g. in terms of Kazhdan-Lusztig-Langlands parameters) the comparison can be explicitly given.
The main tool for making this uniform classification explicit in specific cases is the "Vogan conjecture" which 
enables the explicit computation in terms of $b\in\Bc_{gm}$ of the central character $W_0r_b$. 

Let $b\in\Bc_{gm}$, let $C\subset \Qc_b$ be a connected component (an open cone in $\Qc$), and let 
$\vb_0\in\partial(C)\subset \Qc^{sing}_b$. 
An underlying issue is the behavior of the families $\textup{Ind}_{D}(b;\vb)$ near $\vb_0\in \Qc^{sing}_b$. 
These questions play a technical role in the proof of the above Theorem, and are of independent interest. 
To be sure, the family $\textup{Ind}_{D}(b;\vb)$ is \emph{not} continuous in any neighborhood of $\vb_0$, 
which is one of the reasons that Theorem \ref{thm:uniform}(g) is surprising and noteworthy. 
We will show that there is a sense in which the family $\textup{Ind}_{D}(b;\vb)$ can be extended along smooth 
curves in $\Qc_\Cb$ 
as an algebraic 
family of genuine characters, provided one lifts the condition that the characters in the family are discrete series 
characters.  More precisely, consider an affine smooth curve $\Cf\subset \Qc_\Cb$ which intersects $\overline{C}$ 
in a real curve containing $\vb_0\in\overline{C}$. 
Then $\Cf\cap C\ni\vb\to\textup{Ind}_{D}(b;\vb)$ extends as a rational family 
of generically irreducible genuine characters to a finite branched covering $\tilde\Cf$ of $\Cf$. Importantly, this family is 
unramified at points of $\Cf\cap \overline{C}$, and 
regular at the points of $\tilde\Cf$ lying above $\vb_0$. In particular one can define a limit $\textup{Ind}_{D}(b,C;\vb_0)$
at $\vb_0$ of the family of discrete series along $\Cf$ "from the direction of $C$". 
This "limit of discrete series" is a tempered character, which 
depends on $(C,\vb_0)$. We expect that it does not depend on the choice of $\Cf$. Notice that if $\Cf$ intersects 
the boundary of $C$ transversally at $\vb_0$ then $\Cf$ will also intersect the chamber $C_-$ opposite to $C$ with 
respect to $\vb_0$, and thus there exists also a limit $\textup{Ind}_{D}(b,C_-,\Cf;\vb)$. It would be interesting to investigate  how these two limits are related to each other, in terms of the relevant analytic R-group. 

\section{Massive pure elliptic virtual characters}\label{sec:pure}
\subsection{Elliptic virtual characters of affine Weyl groups}
\subsubsection{Elliptic virtual characters of affine Weyl groups}\label{subsub:ellaff}
We identify $X$ with the normal subgroup $\{e\}\times X\subset W$, 
and $W_0$ with the subgroup $W_0\times\{0\}$.  
Let $E=\mathbb{R}\otimes X$, then $W$ acts naturally on the Euclidean space $E$ 
as a group of affine isometries. 

The lattice $X \subset W$ is the normal subgroup of elements whose conjugacy class is 
finite. A centralizer of an element $x\in X$ is called a \emph{Levi subgroup} of $W$.
There are finitely many Levi subgroups of $W$, and this collection is conjugation invariant. 
Each Levi subgroup $L\subset W$ is itself 
an affine Weyl group $L=W_L\ltimes X$, where $W_L$ is a Levi subgroup of $W_0$ (the 
isotropy group of $x$ in $W_0$). Then $W_L$ is a Coxeter group, and has a unique   
set of simple reflections $S_L$ consisting of reflection $r_\alpha\in W_0$ with  
$\alpha\subset R_{0,+}$. Every Levi subgroup is conjugate to a standard Levi subgroup.
We call $L$ \emph{standard} if $S_L\subset S_0$.
 
An element $w\in W$ is called \emph{elliptic} if $w$ does not belong to any proper 
Levi subgroup $L\subset W$ (see \cite{OpdSol}). The following are easily seen to be equivalent:
\begin{enumerate}
\item[(a)] $w\in W$ is elliptic.  
\item[(b)] The canonical image of $w$ in $W_0$ is elliptic (w.r.t the action of $W_0$ on $E$).
\item[(c)] The centralizer of $w$ in $W$ is finite.
\item[(d)] The conjugacy class of $w$ is a union of left (or equivalently right) 
cosets of a sublattice of $X$ of maximal rank.  
\item[(e)] $w$ has isolated fixed points in $E$.
\item[(f)] $w$ has a unique fixed point in $E$.
\end{enumerate}
The set of elliptic elements is a finite union 
of conjugacy classes.

An (extended) \emph{parahoric subgroup} of $W$ is the pointwise stabilizer 
of an affine subspace in $E$.
Given $e\in E$ consider the isotropy group $W_e\subset W$.  Then $W_e$ has a 
natural faithful linear action on $T_e(T)\simeq E$, and an element $w\in W_e$ is said to be elliptic if $w$ is elliptic in $W_e$ with respect to the action on $T_e(T)$, in the sense of Reeder \cite{Ree}. 

We denote by $\mathcal{R}_\mathbb{Z}(W)$ the Grothendieck group of the category of $\mathbb{C}[W]$-modules
of finite length, and by $\mathcal{R}_\mathbb{C}(W)$ its complexification. The character map defines an embedding 
of $\mathcal{R}_\mathbb{C}(W)$ into the space of complex class functions on $W$. 
Let $\overline{\mathcal{R}}_\mathbb{C}(W)$ denote the complex valued class functions on $W$ 
supported on the set of elliptic conjugacy classes. When we compose the character map with the restriction map  
we obtain a surjective map from $\Rc_\mathbb{C}(W)$ to $\overline{\mathcal{R}}_\mathbb{C}(W)$. 
In \cite{OpdSol} it was shown that the kernel of this map is spanned by the set of 
characters which are induced from proper Levi subgroups. We will identify 
$\overline{\mathcal{R}}_\mathbb{C}(W)$ with this 
quotient of $R_\mathbb{C}(W)$. We denote by $\overline{\mathcal{R}}_\mathbb{Z}(W)\subset\mathcal{R}_\mathbb{C}(W)$
the image of $\mathcal{R}_\mathbb{Z}(W)$, and refer to this lattice as the group of elliptic virtual characters.

There exists a unique conjugation invariant measure \cite[Theorem 3.3(c)]{OpdSol} $\mu_{ell}$ on $W$, which is supported on the elliptic 
conjugacy classes, and which is defined by $\mu_{ell}((1-w)(X))=|W_0|^{-1}$ if $(1-w)(X)$ has maximal rank, 
and $\mu_{ell}((1-w)(X))=0$ otherwise. 
This defines an integral positive semidefinite Hermitian pairing, the \emph{elliptic pairing} $EP_W$ on $R_\mathbb{C}(W)$
by integrating $f\overline{g}$ over $W$ with respect to the measure $\mu_{ell}$. 
The Euler-Poincar\'e pairing on $R_\mathbb{Z}(W)$ is expressed by $EP_W$. 
More precisely \cite{OpdSol}, given 
virtual representations $U$ and $V$ of $W$, with characters $\chi_U$ and $\chi_V$ respectively, one has  
\begin{equation}
EP_W(U,V)=\int_{w\in W}\chi_U(w)\overline{\chi_V(w)}d\mu_{ell}(w)=\sum_{i=0}^\infty(-1)^i\textup{dim}\textup{Ext}_W^i(U,V)
\end{equation}
In particular $EP_W$ is integral on $R_\mathbb{Z}(W)$.
By \cite{OpdSol} the radical of $EP_W$ is exactly the kernel of the quotient map 
$\Rc_\mathbb{C}(W)\to \overline{\mathcal{R}}_\mathbb{C}(W)$, hence in particular $EP_W$ descends to a 
positive definite integral inner 
product on $\overline{\mathcal{R}}_\mathbb{Z}(W)$. 

The Weyl group $W_0$ acts naturally on the algebraic torus $T=\textup{Hom}(X,\mathbb{C}^\times)$.
Clearly $w\in W_0$ is elliptic if and only if $w$ has finitely many fixed points on $T$. It was shown 
in \cite{OpdSol} that   the set of elliptic conjugacy classes of $W$ and the set 
of $W_0$-orbits of pairs $(C,t)$ with $C\in W_t$ an elliptic conjugacy class (with respect to the 
faithful action of $W_t$ on $T_t(T)$) and $t\in T$ have the same cardinality.
Here the action of $W_0$ is defined by $w(C,t)=(wCw^{-1},wt)$. 

Elements $f\in \overline{\mathcal{R}}_\mathbb{C}(W)$ can be viewed as tracial functionals $f\in\mathbb{C}[W]^*$
supported on the set of elliptic conjugacy classes. Hence the center $Z(\mathbb{C}[W])=\mathbb{C}[X]^{W_0}\subset \mathbb{C}[W]$ 
acts on $\overline{\mathcal{R}}_\mathbb{C}(W)$ by multiplication, i.e. $z.f(a):=f(za)$ for all $z\in Z(\mathbb{C}[W])$, 
$f\in \overline{\mathcal{R}}(W)$, 
and $a\in \mathbb{C}[W]$.
Using Mackey theory we showed in \cite{OpdSol}  that there exists an isometric isomorphism 
\begin{equation}\label{eq:Mack}
\textup{Ind}:=\oplus_{s\in W_0\backslash T}\textup{Ind}_s:\oplus_{s\in W_0\backslash T}
\overline{\mathcal{R}}_\mathbb{C}(W_s)\to\overline{\mathcal{R}}_\mathbb{C}(W)
\end{equation}
Here $\overline{\mathcal{R}}_\mathbb{C}(W_s)$
is equipped with the elliptic inner product \cite{Ree} for the isotropy group $W_s$ 
with respect to its natural faithful representation on the tangent space $T_s(T)$ of $T$ at $s$, 
and the direct sum is an orthogonal direct sum. 
Furthermore $\textup{Ind}_s$ is the linear map on $\overline{\mathcal{R}}_\mathbb{C}(W_s)$ realized by 
the Mackey induction functor. The image of $\textup{Ind}_s$ equals 
\begin{equation}\label{eq:eigen}
\textup{Im}(\textup{Ind}_s)= \overline{\mathcal{R}}_\mathbb{C}(W)_{W_0s}, 
\end{equation}
the $Z(\mathbb{C}[W])$-eigenspace in $\overline{\mathcal{R}}_\mathbb{C}(W)$ 
with eigenvalue $W_0s$, hence (\ref{eq:Mack}) gives the orthogonal decomposition of 
$\overline{\mathcal{R}}_\mathbb{C}(W)$ 
as direct sum of $Z(\mathbb{C}[W])$-eigenspaces. 
By Mackey theory for $W=W_0\ltimes X$ this decomposition is compatible with the integral 
structure.
It follows that we also have an orthogonal direct sum decomposition of lattices:
\begin{equation}\label{eq:MackZbar}
\textup{Ind}:=\oplus_{s\in W_0\backslash T}\textup{Ind}_s:\oplus_{s\in W_0\backslash T}
\overline{\mathcal{R}}_\mathbb{Z}(W_s)\to\overline{\mathcal{R}}_\mathbb{Z}(W)
\end{equation}
\begin{defn}[\cite{COT}] Let $\Xc$ be a $\Zb$-lattice equipped with an integral positive 
definite bilinear form. An element $x\in \Xc$ is called \emph{pure} if $x$ is not a nontrivial 
orthogonal sum in  $\Xc$.  
\end{defn}
\subsection{Affine Hecke algebras and Dirac induction}
Unfortunately we do not know how to define a Dirac type operator 
for affine Hecke algebras. Using appropriate versions of Lusztig's reduction theorems, 
results of \cite{COT}, \cite{OpdSol}, and \cite{OpdSol2} we can nevertheless define Dirac 
type induction from a well defined subspace of the space of elliptic characters of the affine 
Weyl group to the space of virtual discrete series characters of $\Hc_\vb$.

For affine Hecke algebras we use the setup and the notations of \cite[Section 2]{OpdSol2}. Thus 
given a based root datum $\mathcal{R}$ let $\Lambda$ denote the canonically associated Laurent  
polynomial ring of Hecke parameters $v(s)$, and let $\Hc_\Lambda=\mathcal{H}_\Lambda(\mathcal{R})$ denote the 
associated affine Hecke algebra defined over $\Lambda$. Let $\Qc_c=\textup{Hom}(\Lambda,\mathbb{C})$, 
the group of complex points of an algebraic torus. Let $\Qc\subset \Qc_c$ be the real vector group of real points.
 
We denote the canonical $\Lambda$-basis of $\Hc_\Lambda$
by $N_w$ (with $w\in W$), where the normalization is such that for affine simple reflections 
$\mathbf{s}\in S$ we have:
\begin{equation}
(N_\mathbf{s}-v(\mathbf{s}))(N_\mathbf{s}+v(\mathbf{s})^{-1})=0
 \end{equation}
An element $\mathbf{v}\in \Qc$ is determined by its coordinates 
$\mathbf{v}(\mathbf{s}):=v(\mathbf{s})(\mathbf{v})$ with $\mathbf{s}\in S$. We denote by 
$\mathcal{H}_{\mathbf{v}}=\mathcal{H}_\Lambda(\mathcal{R})\otimes\mathbb{C}_{\bf{v}}$ 
the corresponding specialized affine Hecke algebra, specialized at $\mathbf{v}$. 

According to the Bernstein-Lusztig-Zelevinski presentation of $\Hc_\Lambda$ we have a unique abelian subalgebra 
$\mathcal{A}=\Cb[\theta_x\mid x\in X]\subset \Hc$ such that $\theta_x=N_x$
if $x\in X$ is dominant. Then $\mathcal{A}\simeq \Cb[X]$, and the center $Z(\Hc_\Lambda)$ is equal to 
$A^{W_0}\simeq\Lambda[X]^{W_0}$. 

Given $\vb\in \Qc$, consider the quotient $\overline{\mc{R}}_\Cb(\Hc_\vb)$ of the complexified Grothendieck 
ring $\mc{R}_\Cb(\Hc_\vb)$ of finite length representations of $\Hc_\vb$ by the subspace generated by the 
properly parabolically induced representations. By \cite{CiuHe}  
this is a finite dimensional complex vector space for all $\vb\in \Qc$. 
Notice also that $\overline{\mc{R}}_\Cb(\Hc_{\vb=1})=\overline{\mc{R}}_\Cb(W)$. The image of the lattice of virtual 
characters is denoted by $\overline{\mc{R}}_\Zb(\Hc_\vb)$. The center $Z(\Hc_\vb)$ acts on 
$\overline{\mc{R}}_\Cb(\Hc_\vb)$, and by 
Schur's lemma we have a decomposition 
\begin{equation}\label{eq:zblocks}
\overline{\mc{R}}_\Zb(\Hc_\vb)=\oplus_{t\in W_0\backslash T}\overline{\mc{R}}_\Zb(\Hc_\vb)_{W_0t}
\end{equation}
If $\vb\in \Qc$ is a positive parameter then 
\cite{OpdSol} asserts that $\Hc_\vb$ has finite global dimension, and we define an integral bilinear 
form $EP_\Hc$ on $\overline{\mc{R}}_\Zb(\Hc_\vb)$ by 
\begin{equation}
EP_\Hc(U,V)=\sum_{i=0}^\infty(-1)^i\textup{dim}\textup{Ext}_{\Hc_\vb}^i(U,V)
\end{equation}
As mentioned above, there exists \cite{OpdSol} a 
``scaling map'' $\lim_{\vb\to 1}: \overline{\mc{R}}_\Cb(\Hc_\vb)\to \overline{\mc{R}}_\Cb(W)$
which is an isometry. In particular $EP_\Hc$ is itself symmetric and positive semidefinite. One way to understand 
$\lim_{\vb\lim_\to 1}$ is via Lusztig's reduction results to graded affine Hecke algebras, combined with Clifford theory, and 
the restriction map from graded affine Hecke algebra representations to representations of the corresponding Weyl group.
This is what we will look into in the next paragraph.

\subsubsection{Clifford theory for extensions of graded affine Hecke algebras}\label{sec:cliff}
Consider $\vb\in \Qc$, and let $V$ be an irreducible representation of $\Hc_\vb$. Recall the polar decomposition $T=T_uT_v$, where $T_u=\textup{Hom}(X,S^1)$ and $T_v=\textup{Hom}(X,\mathbb R_{>0}).$ Let the central character of $V$ be $W_0t$ with  
$t=sc$ where $s\in T_u$ is a unitary element, and $c\in T_v$. Let $F_{s,1}$, $R_{s,1}$ and $\Gamma_s$ be as in 
\cite[Definition 2.5]{OpdSol2}, so that the isotropy group $W_s$ of $s$ in $W_0$ equals $W_s=W(R_{s,1})\rtimes \Gamma_s$, 
and $\alpha(t)>0$ for all $\alpha\in R_{s,1}$. We recall that $\Gamma_s$ is a finite abelian group, acting on 
$F_{s,1}$ by diagram automorphisms  
preserving $k_{s,1}$. Lusztig's reduction theorems \cite{Lus2} in the version discussed in 
\cite[Theorem 2.6, Theorem 2.8, Corollary 2.10]{OpdSol2} imply that the category of 
finite dimensional representations of $\Hc_\vb$ with central character $W_0t$ is equivalent to the category 
of finite dimensional representations of $\Hb(R_{s,1},T_s(T),F_{s,1};k_s)\rtimes\Gamma(t)$ with real central character 
$W(R_{s,1})\xi$. Here $\Hb(R_{s,1},T_s(T),F_{s,1};k_s)$ is the graded affine Hecke algebra as defined in \cite[Section 2]{OpdSol2}, 
$\xi \in T_s(T)$ is the unique vector in the real span of $R_{s,1}^\vee$ such that $\alpha(t) = e^{\alpha(\xi)}$ for all $\alpha\in R_{s,1}$ 
and $\Gamma(t)\subset \Gamma_s$ is the isotropy group of the central character $W(R_{s,1})\xi$ of $\Hb(R_{s,1},T_s(T),F_{s,1};k_s)$.

Clifford theory \cite{RamRam} for the crossed product $\Hb(R_{s,1},T_s(T),F_{s,1};k_s)\rtimes\Gamma_s$ says that the irreducible characters 
of this algebra are obtained as follows. Let $U$ be an irreducible representation of $\Hb(R_{s,1},T_s(T),F_{s,1};k_s)$. Let 
$\Gamma_U\subset \Gamma_s$ be the isotropy subgroup for the equivalence class $[U]$ of irreducible representations 
of $\Hb(R_{s,1},T_s(T),F_{s,1};k_s)$. Then twisting $U$ by elements of $\Gamma_U$ equips $U$ with a representation of 
a twisted group algebra $\mathbb{C}[\Gamma_U;\eta_U]$ with respect to a $2$-cocycle $\eta_U$ of $\Gamma_U$ with values 
in $\mathbb{C}^\times$. Consider  a simple module $M$ of $\mathbb{C}[\Gamma_U;\eta_U^{-1}]$, then 
$N_\Hb(U,M):=\textup{Ind}_{\Hb(R_{s,1},T_s(T),F_{s,1};k_s)\rtimes\Gamma_U}^{\Hb(R_{s,1},T_s(T),F_{s,1};k_s)\rtimes\Gamma_s}(U\otimes M)$
is an irreducible $\Hb(R_{s,1},T_s(T),F_{s,1};k_s)\rtimes\Gamma_s$-module, and all its irreducible modules are equivalent to such 
a module. Moreover, $N_\Hb(U,M)\simeq N_\Hb(U',M')$ if and only if $U'\simeq U\circ\gamma^{-1}$ for some  $\gamma\in \Gamma_s$, 
and $M'\simeq M\circ\gamma^{-1}$. 

Observe that when $U$ has central character $W(R_{s,1})\xi$ then $\Gamma_U\subset \Gamma(t)$. Thus Clifford theory 
implies that the set of 
irreducible modules of $\Hb(R_{s,1},T_s(T),F_{s,1};k_s)\rtimes\Gamma_s$ with central character $W_s\xi$ is in natural 
bijection with the set of irreducible modules of $\Hb(R_{s,1},T_s(T),F_{s,1};k_s)\rtimes\Gamma(t)$ with central character 
$W(R_{s,1})\xi$. In fact it follows from the proof of Lusztig's reduction theorem that this bijection between the respective 
sets of irreducibles arises from a Morita equivalence of the two algebras, formally completed at the 
appropriate central characters $W_s\xi$ and $W(R_{s,1})\xi$ respectively.
Therefore, by the above, the category of finite dimensional representations of $\Hc_\vb$ with central character $W_0t$ 
is naturally equivalent with the category 
of finite dimensional representations of $\Hb(R_{s,1},T_s(T),F_{s,1};k_s)\rtimes\Gamma_s$ with real central character 
$W_s\xi$. In particular we have a natural isomorphism
\begin{equation}\label{eq:red}
\overline{\mc{R}}_\Zb(\Hc_\vb)_{W_0t}\simeq \overline{\mc{R}}_\Zb(\Hb(R_{s,1},T_s(T),F_{s,1};k_s)\rtimes\Gamma_s)_{W_s\xi}
\end{equation}
It is an interesting question what the central support of $\overline{\mc{R}}_\Zb(\Hc_\vb)$ is. Clearly, if 
$\overline{\mc{R}}_\Zb(\Hc_\vb)_{W_0t}\not=0$ then, by the above, one has  
$T_s(E)^{W_s}=0$. 

Since $\Cb[W_s]=\Cb[W(R_{s,1})]\rtimes \Gamma_s$, we have a similar description of the set of 
irreducibles of $\Cb[W_s]$ as modules of the form $N_{W_s}(X,M)$ where $X$ is an irreducible for $W(R_{s,1})$.

The restriction functor 
\[ \textup{Res}_{W_s}: \Hb(R_{s,1},T_s(T),F_{s,1};k_s)\rtimes\Gamma_s\text{-modules}\to \Cb[W_s]=\Cb[W(R_{s,1})]\rtimes \Gamma_s\text{-modules},
\]
 induces a homomorphism on the level of the Grothendieck 
groups of representations of finite length. Via the above correspondences, the ``scaling map" $\lim_{\vb\to 1}$ (more precisely, 
$\lim_{\ep\to 0}\pi_{\vb^\ep}$)  
corresponds to taking the limit $\ep\to 0$ of the family of twists by linear scaling isomorphisms 
$\psi_\ep:\Hb(R_{s,1},T_s(T),F_{s,1};k_s)\rtimes\Gamma_s\to \Hb(R_{s,1},T_s(T),F_{s,1};\ep k_s)\rtimes\Gamma_s$ defined by 
$\phi_\ep(\xi)=\ep^{-1} \xi$. This is   
the restriction map $\textup{Res}_{W_s}$. In particular we see: 

\begin{cor}\label{cor:lim}
The map $[\pi]\to\lim_{\ep\to 0}[\pi_{\vb^\ep}]$ respects the lattices of virtual characters, and defines
an isometric map $\overline{\mc{R}}_\Zb(\Hc_\vb)\to \overline{\mc{R}}_\Zb(W)$ sending 
$\overline{\mc{R}}_\Zb(\Hc_\vb)_{W_0t}$ to $\overline{\mc{R}}_\Zb(W)_{W_0s}$ (where $t=sc\in T_uT_v$ as before).
More precisely, if $\pi$ corresponds to the module $U$ of $\Hb(R_{s,1},T_s(T),F_{s,1};k_s)\rtimes\Gamma_s$ via 
(\ref{eq:red}), and $b_s=[U|_{W_s}]$, then $\lim_{\ep\to 0}[\pi_{\vb^\ep}]=\textup{Ind}_s(b_s)$.
\end{cor}
\begin{defn}\label{prop:Ygm}
We define 
\begin{equation}\label{eq:Y}
\Yc_\vb:=\lim_{\vb\to 1}(\overline{\mc{R}}_\Zb(\Hc_{\vb}))\subset \overline{\mc{R}}_\Zb(W)
\end{equation}
and let $\Yc_{\vb-m}\subset\Yc_\vb$ be the sublattice generated by the limits $\lim_{\epsilon\to 0}\pi_{\vb^\ep}$ of 
\emph{discrete series representations} of $\Hc_\vb$. 
We call $\Yc_{\vb-m}$ the lattice of $\vb$-massive elliptic representations of $W$. 
Finally let $\Yc_{gm}\subset \overline{\mc{R}}_\Zb(W)$ be the sublattice generated by $\cup_{\vb\in\Qc}\Yc_{\vb-m}$, the lattice 
of \emph{generically massive} elliptic representations of $W$.  In general, $\Yc_{gm}\neq \overline{\mc{R}}_\Zb(W)$.
\end{defn}

\begin{prop}\label{prop:Bgm}
The lattice $\Yc_{gm}$ admits an orthonormal basis $\Bc_{gm}\subset \Yc_{gm}$. 
If $b\in\Bc_{gm}$ then $b\in\overline{\mc{R}}_\Zb(W)_{W_0s}$ for some $s\in T_u$ such that 
$\textup{rk}(R_{s,1})=\textup{dim}(T_u)$. 
\end{prop}
\begin{proof}
It follows from the classification \cite[Section 5]{OpdSol2} that an irreducible 
discrete series $(V,\pi_\vb)$ of $\Hc_\vb$ has central character $W_0r(\vb)$ for some generic residual 
point $r$, and we can write $r=sc$ with $s\in T_u$ and $c=\exp(\xi)$ where $\xi$ is a generic linear residual 
point for $\Hb(R_{s,1},T_s(T),F_{s,1};k_s)$ whose evaluation at $k_s$ is residual. In particular, the rank of 
$R_{s,1}$ is equal to the dimension of $T_u$. 

By \cite{OpdSol2} it also follows that $\pi$ corresponds 
via (\ref{eq:red}) to the induced representation of an irreducible representation 
$U\otimes M$ of 
$\Hb(R_{s,1},T_s(T),F_{s,1};k_s)\rtimes\Gamma_U$ to $\Hb(R_{s,1},T_s(T),F_{s,1},k_s)\rtimes\Gamma_s$, 
where $U$ is an irreducible discrete series character of $\Hb(R_{s,1},T_s(T),F_{s,1},k_s)$ (here we also use the discussion in the text above). 
In Corollary \ref{cor:lim} we have seen that the limit 
$b:=\lim_{\epsilon\to 0}\overline{\pi_{\vb^\epsilon}}$ equals $b=\textup{Ind}_s(b_s)$ with $b_s=\overline{U|_{W_s}}$. 
By the results of 
\cite{COT} and of \cite{OpdSol2} we see that the algebraic Dirac induction for $\Hb(R_{s,1},T_s(T),F_{s,1};k_s)$  of $b_s$ 
yields an elliptic representation 
supported by the central character $W(R_{s,1})\xi(k_s)$, which is (by \cite{OpdSol2}) residual for all $\vb\in\Qc^{reg}_b$, a 
complement in  $\Qc$ of finitely many hyperplanes. Recall that, by \cite{Opd3}, a central character $W(R_{s,1})\xi(k_s)$ for $\Hb(R_{s,1},T_s(T),F_{s,1};k_s)$ is residual if and only if $W_s\xi(k_s)$ is residual for  $\Hb(R_{s,1},T_s(T),F_{s,1};k_s)\rtimes\Gamma_s$. (This is because of the invariance of residual central character under Dynkin diagram automorphims.)

Let $\vb'\in\Qc^{reg}_b$. 
By the main result of \cite{COT}, since $W_s\xi(k'_s)$ is residual there 
exists a virtual discrete series character $U'$ of $\Hb(R_{s,1},T_s(T),F_{s,1};k'_s)$ with $\overline{U'|_{W_s}}=b_s$
($U'$ is the analytic Dirac induction of $b_s$). 
Again using the classification 
of the discrete series of \cite{OpdSol2}, and (\ref{eq:red}), there exists a virtual discrete series character $\pi'_{\vb'}$ 
of $\Hc_{\vb'}$ with $b=\textup{Ind}_s(b_s)=\lim_{\epsilon\to 0}\overline{\pi_{{\vb'}^\epsilon}}$. 
Since $\Yc_{gm}\subset \overline{\mc{R}}_\Zb(W)$ it is clear that $\Yc_{gm}$ is finitely generated. Choose a 
finite collection of (irreducible) discrete series $\pi_i$ of $\Hc_{\vb_i}$ such that the corresponding 
limits $b_i$ are linearly independent and generate $\Yc_{gm}$.
By the arguments above, if $\vb\in \cap_i\Qc^{reg}_{b_i}$ then there exist virtual discrete series characters 
$\pi'_i$ of $\Hc_{\vb}$ with $b_i:=\lim_{\epsilon\to 0}\overline{\pi'_{i,\vb^\epsilon}}$. Consequently, $\Yc_{gm}=\Yc_{\vb-m}$.
Since the limit map is an isometry and since  (by \cite{OpdSol}) the irreducible discrete series form an orthonormal set with 
respect to 
$EP_\Hc$, $\Yc_{\vb-m}$ (and thus $\Yc_{gm}$) admits an orthonormal basis.
\end{proof}

At this point, the basis $\Bc_{gm}$ from Proposition \ref{prop:Bgm} is not unique. The canonical choice for the basis $\Bc_{gm}$ is obtained in Corollay \ref{cor:unique}. 

\subsubsection{Dirac induction for affine Hecke algebras}\label{subsec:dirac}
\begin{defn}\label{defn:indD}
Given $b\in\Bc_{gm}$ and $\vb\in \Qc_b^{reg}$, we define $\textup{Ind}_D(b;\vb)$ (the "Dirac induction of $b$")  
as the unique virtual discrete series character of $\Hc_\vb$ whose scaling limit satisfies 
$\lim_{\ep\to 0}\textup{Ind}_D(b;\vb^\ep)=b$. Up to a sign $\ep(b;\vb)\in\{\pm 1\}$,  $\textup{Ind}_D(b;\vb)$
is an irreducible discrete series character. For all $\vb\in\Qc$ this defines a bijection 
$\Bc_{\vb-m}\ni b\to \ep(b;\vb)\textup{Ind}(b;\vb)\in\Delta_\vb$, where $\Delta_\vb:=\Delta(\Hc_\vb)$ denotes 
the set of isomorphism classes of 
irreducible discrete series representations of $\Hc_\vb$.  
\end{defn}
In the proof of Proposition \ref{prop:Bgm}
we have seen that $\textup{Ind}_D(b;\vb)$ is not directly constructed as the index of a Dirac type operator but rather, 
$\textup{Ind}_D(b;\vb)$ corresponds via (\ref{eq:red}) and a Morita equivalence to the discrete series $U\otimes M$ of 
$\Hb(R_{s,1},T_s(T),F_{s,1},k_s)\rtimes\Gamma_U$, with $U=\textup{Ind}_D^{an}(b_s,k_s)$ the analytic Dirac induction 
(defined in \cite{COT}) 
of $b_s$ for  $\Hb(R_{s,1},T_s(T),F_{s,1},k_s)$. The existence of $\textup{Ind}_D(b;\vb)$ follows from this construction. 
It also follows 
from this construction that $\pm\textup{Ind}_D(b;\vb)$ is irreducible. 
\subsubsection{The generic Vogan central character map}
Let $\Bc_{gm}$ be an orthonormal basis of $\Yc_{gm}$. As we have seen, given 
$b\in\Bc_{gm}$ the set $\Qc_b^{reg}= \{\vb\mid b\in\Bc_{\vb-m}\}$ is the complement of finitely many 
hyperplanes in $\Qc$. As in the proof of Proposition \ref{prop:Bgm}, to each $b\in\Bc_{gm}$ we have 
a canonically associated orbit of generic 
residual points $W_0r_b$, with $r_b=s\exp{\xi_s}$, and $\xi_s$ the generic linear residual point associated to 
$b_s\in \overline{R_\Zb(W_s)}$ by the generic version of ``Vogan's conjecture" (see \cite[Theorem 3.2]{COT}). Strictly speaking, the results of \cite{COT} apply to give a residual central character of a nonextended graded affine Hecke algebra, but as already noted in the proof of Proposition \ref{prop:Bgm}, one may use the invariance \cite{Opd3} of these central characters under diagram automorphisms to obtain the desired central character in our more general setting.

Let us denote the resulting generic central character map 
(cf. \cite{OpdSol2}) by: 
\begin{align*}
gcc_B:\Bc_{gm}&\to W_0\backslash\textup{Hom}(X,\Lambda^\times)=W_0\backslash T(\Lambda)\\
b&\to W_0(s\exp{\xi_s})
\end{align*}
(the generic Vogan central character map).
From the proof of Proposition \ref{prop:Bgm} and Definition \ref{defn:indD} it is clear that:
\begin{cor}\label{cor:gVcc}
For all $b\in\Bc_{gm}$ and $\vb\in\Qc_b^{reg}$ we have 
$$gcc_\vb(\textup{Ind}_D(b;\vb))=gcc_B(b)=W_0r_b$$ 
Moreover, $W_0r_b(1)=W_0s$ if and only if $b$ can be written as  $b=\textup{Ind}_s(b_s)$.
\end{cor}
We put $m_b:=m_{W_0r_b}$, where $m_{W_0r_b}$ denotes the mass function associated to the 
orbit of generic residual points $gcc_B(b)=W_0r_b$ in \cite{OpdSol2}. 
By \cite{OpdSol2}, this is an explicit rational function on $\Qc$ 
which is regular on $\Qc$, and we have for all $b\in\Bc$ that $\Qc_b^{reg}=\{\vb\in\Qc\mid m_b(\vb)\not=0\}$. 
\begin{thm}\label{thm:Vogan}
The generic Vogan central character map $gcc_B$ is a surjection 
$gcc_B:\Bc_{gm}\to W_0\backslash\textup{Res}(\Rc)$. 
\end{thm}
\begin{proof} In \cite{Opd1} it was shown (also see \cite{OpdSol2}) 
that for any $\vb\in\Qc$, every orbit of residual points  
of $\Hc_\vb$ is the specialization at $\vb$ of a generic orbit of residual points $W_0r$ at $\vb$. 
The main theorem of \cite{Opd1} states that for any $\vb\in\Qc$, and any  
orbit of residual cosets $W_0r(\vb)$, there exists a discrete series character $\pi$ of $\Hc_\vb$ such that 
$W_0r(\vb)$ is the central character of $\pi$. The set $W_0\backslash\textup{Res}(\Rc)$ is finite, hence 
$\Qc^{gen}:=\{\vb\in \cap_{W_0r\in W_0\backslash\textup{Res}(\Rc)}\Qc^{reg}_{W_0r}\mid 
\mathrm{For\ all\ }r,r'\in\textup{Res}(\Rc):\ W_0r(\vb)=W_0r'(\vb)\mathrm{\ only\ if\ }Wr=W_0r'\}$ 
is open, dense in $\Qc$. 
Let $\vb\in\Qc^{gen}$ and $W_0r\in W_0\backslash\textup{Res}(\Rc)$. By the above there exists a 
$\pi\in \Delta_\vb$ such that $cc(\pi)=W_0r(\vb)$. Since $\vb\in\Qc^{gen}$ this implies that 
$gcc_\vb(\pi)=W_0r$.  
(For the definition of the generic central character map $gcc_\vb$: See \cite[Definition 5.4]{OpdSol2}.)
Put $b:=\lim_{\ep\to 0}\pi_{\vb^\ep}$. By Corollary \ref{cor:gVcc} we have $gcc_B(b)=W_0r$.
\end{proof}
Denote by $\Bc_{W_0r}\subset \Bc_{gm}$ 
the fibre $gcc_\Bc^{-1}(W_0r)$ of the map from Theorem \ref{thm:Vogan}.

\section{One-dimensional algebraic families of discrete series representations, and their limits}\label{sec:familiesandlimits}
\subsection{One-dimensional algebraic families of discrete series}\label{sec:fan}
Given a generic residual point $r\in\textup{Hom}(X,\Lambda^\times)$ we know that 
$r(\vb)$ is a residual point for all $\vb\in\Qc_{W_0r}^{reg}$, which is the complement of finitely many 
hyperplanes in $\Qc$. By \cite[Corollary 5.9]{OpdSol2}) we have a nonempty set of 
irreducible discrete series characters of $\Hc_\vb$ with \emph{generic central character} $W_0r$. 
Denote this nonempty set by:  
$$\textup{DS}_{W_0r,\vb}:=\{\pi\in\Delta(\Hc_\vb)\mid gcc(\pi)=b\}$$
Hence $W_0r(\vb)$ is the central character of the following nonempty union: 
$$\textup{DS}_{W_0r(\vb)}:=\cup_{\{W_0r'\mid W_0r'(\vb)=W_0r(\vb)\}}\textup{DS}_{W_0r',\vb}$$

Let $\vb\in\Qc_\Cb$. For each $r'(\vb)\in W_0r(\vb)$, choose a convex open neighborhood 
$U_{r'(\vb)}$ of $r'(\vb)\in T$ such that the only residual cosets of the 
$\mu$-function $\mu_{\Rc,\vb}$ of $\Hc_\vb$ which intersect $U_{r'(\vb)}$ in fact contain $r'(\vb)$.
Let $U_{r'(\vb)}^{reg}=U_{r'(\vb)}\cap T^{reg}$ be the complement in $U_{r'(\vb)}$ of the union of the set of residual cosets  
of $\mu_{\Rc,\vb}$ in $T$. The choice of $U_{r'(\vb)}$ as above implies that 
$U_{r'(\vb)}^{reg}$ is homeomorphic to the complement of a central hyperplane arrangement in 
a complex vector space with origin $r'(\vb)$. In particular, the homology group 
$H_n(U_{r'(\vb)}^{reg},\Zb)$ only depends 
on the local structure of the pole hyperplane arrangement at $r'(\vb)$ and if we would shrink $U_{r'(\vb)}$ 
to a smaller convex open neighborhood $U_{r'(\vb)}^{\prime,reg}\subset U_{r'(\vb)}$ of $r'(\vb)$ 
this would induce a canonical 
isomorphism $H_n(U_{r'(\vb)}^{\prime,reg},\Zb)=H_n(U_{r'(\vb)}^{reg},\Zb)$.
Let us denote the direct limit $\varinjlim H_n(U_{r'(\vb)}^{\prime,reg},\Zb)$ by $H_{n,r'(\vb)}(\Zb)$.

In general, if $U\subset T$ is homeomorphic to a complex ball $B\subset \mathfrak{t}$, 
via the exponential mapping of the complex algebraic torus $T$, then $U^{reg}:=U\cap T^{reg}$
is homeomorphic to the intersection of $B$ with the complement of an affine hyperplane arrangement.  
In this section we will need some basic facts about the topology of hyperplane 
arrangements, cf. \cite{OT}, \cite{SV}.
By \cite[Paragraph 4.4]{SV} it easily follows that $H_n(U^{reg},\Zb)$ is 
in a canonical 
way a \emph{direct sum} of the $H_{n,p}(\Zb)$ where $p$ runs over the set of points of $U$ which lie in the  
intersection lattice generated by the codimension one residual cosets of $\mu_{\Rc,\vb}$. Let 
$\pi_p:H_n(U^{reg},\Zb)\to H_{n,p}(\Zb)$ denote the corresponding projection.
 
Let us denote the collection of open sets $\{U_{r'(\vb)}\}_{r'\in W_0r}$ by $\Uc$.
Let $\Oc_{\Uc,\vb}\subset \Qc_\Cb$ be an open ball with center $\vb$ 
with the property that for all $\vb'\in \Oc_{\Uc,\vb}$ and all $r'\in W_0r$ we have $r'(\vb')\in U_{r'(\vb)}$.
Given $\vb\in\Qc_\Cb$ and a homology class $[\xi_{r'(\vb)}]\in  H_n(U_{r'(\vb)},\Zb)$ 
for each $r'(\vb)\in W_0r(\vb)$, this defines for all $\vb'\in \Oc_{\Uc,\vb}$ a unique 
class $\pi_{r'(\vb')}([\xi_{r'(\vb')}])\in H_{n,r'(\vb')}(\Zb)$. It is easy to see that this procedure defines a topology basis of
the \'etale space of a sheaf $\Fc^H_{n,r'}$ of abelian groups over $\Qc_\Cb$,  
with stalks $H_{n,r'(\vb)}(\Zb)$.  Let $\Qc_{W_0r,\Cb}^{gen}\subset \Qc_\Cb$ be the Zariski open set of 
$\vb\in\Qc_\Cb$ such that $|W_0r(\vb')|$ is locally constant at $\vb$, and such that if $r'$ is a  
generic residual point such that $W_0r(\vb)=W_0r'(\vb)$ then $W_0r=W_0r'$.
\footnote{Using the results of \cite{Opd3} it is easy to see that $\Qc^{gen}_{W_0r,\Cb}\cap \Qc=\Qc^{gen}_{W_0r}$, in the notation of \cite{OpdSol2}.}
Clearly, the sheaf $\Fc^H_{n,r}$ is locally trivial in the analytic topology on $\vb\in \Qc_{W_0r,\Cb}^{gen}$. We have shown: 
\begin{lem}\label{lem:sheaf}
For each generic residual point $r$, the homology groups $H_{n,r(\vb)}(\Zb)$ ($\vb\in\Qc_\Cb$) are the stalks of a sheaf $\Fc^H_{n,r}$ (in the analytic topology) 
of finitely generated abelian groups on $\Qc_\Cb$, which is locally trivial on the Zariski open set $\Qc_{W_0r,\Cb}^{gen}$.  
\end{lem}

The main results of \cite{Opd1} and of \cite{OpdSol2} show that for each 
$\vb\in\Qc^{reg}_{W_0r}$ we can choose, for each $r'(\vb)\in W_0r(\vb)$, 
classes $\xi_{r'(\vb)}\in H_{n,r'(\vb)}(\Zb)$ and, for each 
$\chi\in\textup{DS}_{W_0r(\vb)}$, constants $c_{\chi,C}\in\Qb_+$ which only depend on the 
connected component $C=C_{\chi,\vb}\subset \Qc_{\textup{gcc}_\vb(\chi)}^{reg}$ to which $\vb$ belongs
(where $\textup{gcc}_\vb(\chi)=W_0r''$ denotes the generic central character \cite[Definition 5.4]{Opd1} 
of $\chi$), 
such that for all $h\in\Hc_\vb$:
\begin{equation}\label{eq:dschar}
\sum_{\chi\in \textup{DS}_{W_0r(\vb)}}c_{\chi,C}\chi_\vb(h)=
m_{W_0r}(\vb)^{-1}\sum_{r'(\vb)\in W_0r(\vb)} \int_{t\in \xi_{r'(\vb)}}K_\vb(h,t)
 \end{equation}
where $K_\vb(h,t)=E_t(\vb;h)\Delta(t)^{-1}\mu_{\Rc,\vb}dt$ 
is a rational $(n,0)$-form in $t\in T$, with 
$\Delta(t):= \prod_{\alpha>0}(1-\alpha(t)^{-1})$ and with 
$E_t(\vb;h)$ a 
linear functional in $h\in\Hc_\Lambda$, which is regular on $(t,\vb)\in T\times \Qc_\Cb$,  
and such that $E_t(\vb;a)=a(t)\Delta(t)$ for all $a\in \Ac$.

We now prove a fundamental continuity property of (\ref{eq:dschar}) 
(or even the full Plancherel decomposition of the trace of $\Hc_\vb$)  
with respect to the topology of the sheaves $\Fc^H_{n,r}$.
\begin{lem}\label{lem:sect} Let $r$ be a generic residual point 
and let $\vb\in\Qc_{W_0r}^{reg}$. For $r'\in W_0r$ 
and for $\vb'$ in a sufficiently small neighborhood of $\vb$, the 
local homology classes $[\xi_{r'(\vb')}]\in H_{n,r'(\vb')}(\Zb)$ in (\ref{eq:dschar}) 
form a local section of $\Fc^H_{n,r'}$.  
\end{lem}
\begin{proof}
We use the analogous construction as above of  
sheaves $\Fc^H_{n,r_L}(\Zb)$ for all $r_L\in T_L$ a generic residual point. 
Recall the residue lemma \cite[Lemma 3.4]{Opd1}. By \cite[Proposition 3.7]{Opd1} we can 
realize the local traces $\Xf_c$ defined in \cite[Lemma 3.4]{Opd1}  "explicitly" by a system of local cycles 
$\xi_{L(\vb)}\subset \Bc_{L(\vb)}(r_L(\vb),\delta)\subset T_{L(\vb)}$ 
(for all quasi-residual 
cosets $L(\vb)\subset T$) such that for $t_0\in T_v$ deep in the negative Weyl chamber, the cycles 
$t_0T_u$ and $\cup_L \xi_{L(\vb)}\times T^L_u$ are homologous. 
Moreover, the $\xi_{L(\vb)}$ 
satisfy certain local properties (cf. \cite[proposition 3.7]{Opd1}, 
items (i), (ii) and (iii)) 
which guarantee that for all $h\in\Hc_\vb$, the functional $\mathbb{C}[T]\ni f\to\int_{\xi_{L(\vb)}\times T^L_u}f(t)K_\vb(h,t)$ can be written as 
a distribution on $r_L(\vb)T_u$ with support in $L(\vb)^{temp}=r_L(\vb)T^L_u$, applied to $f$ (which is zero if $L_\vb$ is not a residual coset).  

The definition of the local sections of the sheaves $\Fc^H_{n,r_{L'}}$ is such that the local properties ((i), (ii) and (iii) mentioned 
in \cite[Proposition 3.7]{Opd1}) can be satisfied by cycles $\xi_{L'(\vb')}\subset \Bc_{L'(\vb')}(r_{L'}(\vb'),\delta)\subset T_{L'(\vb')}$ 
representing the classes $[\xi_{L'(\vb')}]\in H_{n,r_{L'}(\vb')}(\Zb)$ for $\vb'$ in a small neighborhood of $\vb$.
Moreover, it is automatic that $t_0T_u$ is homologous to the union  of $\cup_{L} \xi_{L(\vb')}\times T^L_u$ (union over 
the generic residual cosets $L$) with a 
collection of cycles of the form $ \xi_{L'(\vb')}\times T^{L'}_u$ where $L'(\vb')$ is quasi-residual (which are irrelevant). 
This implies the result.
\end{proof}
\begin{cor}\label{cor:cont}
For $\vb\in\Qc^{reg}_{W_0r}$, the left hand side of (\ref{eq:dschar}) extends to a  
central functional on $\Hc_{\vb'}$ for all $\vb'$ in a Zariski open neighborhood of $\vb$ 
in $\Qc_\Cb$. 
For $h\in\Hc_\Lambda$ this functional takes values in the quotient field $Q_\Lambda$ of $\Lambda$.
The values of this functional restrict to continuous functions on the connected component $C$ of 
$\vb$ in $\cap_{\chi\in \textup{DS}_{W_0r(\vb)}}\Qc^{reg}_{\textup{gcc}_\vb(\chi)}$. 
\end{cor}
\begin{proof}
We extend the left hand side in a neighborhood of $\vb$ by taking a local sections of 
classes $[\xi_{r'(\vb)}]$ in $\Fc^H_{n,r'}$, and use the right hand side of (\ref{eq:dschar}) to 
define the left hand side. We know from \cite{OpdSol2} that the left hand side extends, for all $h\in\Hc_\Lambda$, 
continuously on $\cap_{\chi\in \textup{DS}_{W_0r(\vb)}}\Qc^{reg}_{\textup{gcc}_\vb(\chi)}$. 
Let $\chi\in\textup{DS}_{W_0r(\vb)}$ and let $\textup{gcc}_\vb(\chi):=W_0r'$. 
Let $\vb'\in \cap_{\chi\in\textup{DS}_{W_0r(\vb)}}C_{\chi,\vb}\cap \Qc_{\textup{gcc}_\vb(\chi)}^{gen}$.
Then the left hand side of (\ref{eq:dschar}) can be written as $\sum_{\{W_0r'\mid W_0r'(\vb)= W_0r(\vb)\}}
\Sigma_{W_0r'}(h;\vb')$  
with: 
\begin{equation}\label{eq:stable}
\Sigma_{W_0r'}(h;\vb'):=\sum_{\chi\in \textup{DS}_{W_0r',\vb'}}c_{\chi,C}\chi_{\vb'}(h) 
\end{equation}
We have $$\Sigma_{W_0r'}(h;\vb')=m_{W_0r'}(\vb')^{-1}\sum_{r''(\vb')\in W_0r'(\vb')} \int_{t\in \xi_{r''(\vb')}}K_{\vb'}(h,t),$$ 
where, for $\vb'$ in a sufficiently small neighborhood of $\vb$,  
the class $[\xi_{r''(\vb')}]$ is equal to the image $\pi_{r''(\vb')}(\xi_{r''(\vb)})\in H_{n,r''(\vb')}(\Zb)$ 
of the constant cycle $\xi_{r''(\vb)}$, 
by Lemma \ref{lem:sheaf} and Lemma \ref{lem:sect}.
It follows that in a small analytic neighborhood of $\vb$ the expression $\Sigma_{W_0r'}(h;\vb'')$ is given by  
an element of $Q_\Lambda$. This extends $\vb'\to \Sigma_{W_0r'}(h;\vb')$ to a rational function on a 
Zariski-open set of $\Qc_\Cb$. 

The continuity of the values of $\Sigma_{W_0r'}(h;\vb')$ in the connected component $C\subset \Qc^{reg}_{W_0r'}$ 
of $\vb'$ follows from 
\cite[theorem 3.4, Proposition 3.7]{OpdSol2}, and this implies the continuity assertion in the theorem.
\end{proof}
\begin{rem} 
We note that $\Sigma_{W_0r}(h;\cdot)$ is not rational  on 
$\cap_{\chi\in \textup{DS}_{W_0r(\vb)}}\Qc^{reg}_{\textup{gcc}_\vb(\chi)}$. 
The rational functions of Corollary \ref{cor:cont} depend on the connected component 
of $\vb$ in $\cap_{\chi\in \textup{DS}_{W_0r(\vb)}}\Qc^{reg}_{\textup{gcc}_\vb(\chi)}$. 
\end{rem}
Let $\Cf\subset \Qc$ be (the set of complex points of) an irreducible real algebraic curve which is 
not contained in $\cup_{W_0r}\Qc^{sing}_{W_0r}$. 
Let $\Rf$ be the ring of regular functions of $\Cf$, with quotient field 
$\Kf$. Thus $\Rf$ is a quotient of $\Lambda$, and we can form $\Hc_\Rf:=\Hc_\Lambda\otimes_\Lambda\Rf$. 
Let $P=\vb_0\in\Cf$ be a smooth point of $\Cf$, and let $C\subset \Qc^{reg}_{W_0r}$ be a connected 
component such that $\vb_0\in\overline{C}$.  Let $\vb'\in C\cap \Qc^{gen}_{W_0r}$, and consider 
$\Sigma_{W_0r}(h;\vb')$. According to Corollary \ref{cor:cont}, this positive rational linear combination 
of the discrete series characters with central character $W_0r(\vb')$  extends  as a rational function in $\vb'$ 
which is continuous on $C$. 
By Corollary \ref{cor:cont} we can restrict the values of $\Sigma_{W_0r}(h;\vb')$ to $\vb'\in \Cf\cap C$. By 
the argument of \cite[Proposition 7]{KoNow} it is clear that this restriction defines a rational function on $\Cf$ 
which we will denote by 
$\Sigma_{W_0r,C}^\Cf(h)\in\Kf$, depending linearly on $h\in\Hc_\Rf$.  
These rational functions are continuous on $\Cf\cap C$. 
\begin{lem}\label{lem:cont} Assume that $\Cf\cap \overline{C}$ consists of smooth points and let 
$P=\vb_0\in\Cf\cap \overline{C}$. 
For all $h\in\Hc_\Rf$ we have:   The rational function 
$\Sigma_{W_0r,C}^\Cf(h)$ is regular in $P=\vb_0$. There exists a Zariski open neighborhood $\Uf\subset \Cf$
of $\Cf\cap\overline{C}$ such that $\Sigma_{W_0r,C}^\Cf(h)\in \Rfu$ for all $h\in\Hc_\Rfu$. 
\end{lem}
\begin{proof}
Since $\Sigma_{W_0r,C}^\Cf(h)$ is rational and $P$ is a smooth point of the curve $\Cf$, 
it is enough to show that $\vb\to \Sigma_{W_0r,C}^\Cf(h;\vb)$ is bounded in a neighborhood of $\vb_0=P\in \Cf$ 
for $\vb\in\Cf\cap C$. But $\Sigma_{W_0r}(\cdot;\vb)$ is a \emph{positive} functional on $\Hc_\vb$ for all $\vb\in C$, and if 
$\vb\in \Cf\cap C$ then 
$\Sigma_{W_0r}(T_e;\vb)=\sum_{\chi\in \textup{DS}_{W_0r,\vb}}c_{\chi,C}\textup{deg}(\chi)=C\in\Qb_+$ is constant.
By \cite[Corollary 2.17(i)]{Opd1} we see that $|\Sigma_{W_0r,C}^\Cf(h;\vb)|\leq C\Vert h\Vert_o$. This implies the regularity 
at $P$ of  the rational function $\vb\to \Sigma_{W_0r,C}^\Cf(h;\vb)$ for all $h\in\Hc_\Rf$. Note that for $z\in\Zc\subset\Hc_\Rf$ 
we have $\Sigma_{W_0r,C}^\Cf(zh;\vb)=z(W_0r(\vb))\Sigma_{W_0r,C}^\Cf(h;\vb)$, while 
$\vb\to z(W_0r(\vb))\in\Rf$ is regular on $\Cf$.
Since $\Hc_\Rf$ is finitely generated over $\Zc$ it follows that we can find an open neighborhood $\Uf$ of  $\Cf\cap \overline{C}$
on which  $\vb\to \Sigma_{W_0r,C}^\Cf(h;\vb)$ is regular for all $h\in\Hc_\Rf$.
\end{proof}
From now on we will assume that $\Cf\cap \overline{C}$ consists of smooth points of $\Cf$, and that $\Uf\subset \Cf$ is as in 
Lemma \ref{lem:cont}.
We have shown that $\Sigma_{W_0r,C}^\Cf\in \Hc_\Rfu^*:=\textup{Hom}_\Rfu(\Hc_\Rfu,\Rfu)$ is a central functional in the 
$\Hc_\Rfu\otimes_\Rfu\Hc_\Rfu^{op}$-module $\Hc_\Rfu^*$,  supported by the central  
character $W_0r$. 
\begin{defn}
$M_{W_0r,C}^\Uf\subset \Hc_\Rfu^*$ denotes the $\Hc_\Rfu\otimes\Hc_\Rfu^{op}$-submodule 
generated by $\Sigma_{W_0r,C}^\Cf$. 
\end{defn}
We will now show that a generic family $\chi$ of discrete series characters as discussed in \cite{OpdSol2} always 
admits  
an algebraic model when we restrict the parameter space to a curve $\Cf$ in $\Qc$. Moreover, we can find such 
models that are regular at the intersection of the curve with the boundary of the connected component 
$C\subset  \Qc_{gcc(\chi)}^{reg}$ (an open cone) on which $\chi$ lives. More precisely:  
\begin{prop}\label{prop:basicM}
\begin{enumerate}
\item[(i)] $M_{W_0r,C}^\Uf$ is  a $\Hc_\Rfu\otimes_\Rfu\Hc_\Rfu^{op}$-module.
The central subalgebra $\Zc_\Rfu\otimes_\Rfu\Zc_\Rfu$ 
acts via scalar multiplication via the character $(z_1\otimes z_2)\to (z_1z_2)(W_0r)\in\Rfu$. 
\item[(ii)] $M_{W_0r,C}^\Uf$ is a locally free $\Rfu$-module of finite rank.
\item[(iii)] $M_{W_0r,C}^\Uf$ has a canonical $\Hc_\Rfu$-algebra structure such that, if we put   
$e=m_{W_0r}\Sigma_{W_0r,C}^\Cf\in M_{W_0r,C}^\Uf\subset \Hc^*_\Rfu$, the map $\Hc_\Rfu\ni h\to he$
is an algebra homomorphism.
\item[(iv)] If $\vb\in\Uf\cap C$ then 
$$M_{W_0r,C}^\Uf\otimes \Cb_\vb\simeq \oplus_{\chi\in \textup{DS}_{W_0r,\vb}}\textup{End}_{\Cb}(V_{\chi_\vb})$$
where $V_{\chi_\vb}$ is a vector space on which the discrete series character $\chi_\vb$ is realized.
\item[(v)] We may assume without loss of generality that $\Uf$ is smooth. There exists a Zariski-open subset 
$\Ufst\subset \Uf$ with $\Cf\cap C\subset \Uf^*$ such that $M_{W_0r,C}^\Ufst:=M_{W_0r,C}^\Uf \otimes_\Rfu \Rfust$ 
is a locally free separable $\Rfust$-algebra of finite rank.
\item[(vi)] 
There exists a branched covering $\phi:\Uft\to\Uf$ such that the following holds true.  Let $\Lf\supset \Kf$ denote 
the function field of $\Uft$, and let $\Rfut:=\Rf(\Uft)$ be the integral closure of $\Rfu$ in $\Lf$. Put 
$\tilde{M}_{W_0r,C}^\Uf:=M_{W_0r,C}^\Uf\otimes\Rfut$.  
There exists a finite set of projective $\Rfut$-modules  of finite rank $\{L_i\}_{i\in\Ic}$, such that 
$$\tilde{M}_{W_0r,C}^\Uf\subset \tilde{M}_{W_0r,C}^{\Uf,max}\simeq \bigoplus_{i\in\Ic}\textup{End}_{\Rfut}(L_i)$$
is a $\Rfut$-order. This turns $L_i$ into a $\Hc_\Rfut$-module, for each $i\in\Ic$.
\item[(vii)] 
With $\Ufst$ as in (v), define $\Uftst:=\phi^{-1}(\Ufst)$ and let $\Rfutst\subset \Lf$ be the integral closure of $\Rfust$
(which is the localization of $\Rfust$ at $\Uftst$).  
Put $\tilde{M}_{W_0r,C}^\Ufst:=M_{W_0r,C}^\Ufst\otimes_\Rfust\Rfutst$. 
Then we have an isomorphism
$$\tilde{M}_{W_0r,C}^\Ufst = \tilde{M}_{W_0r,C}^{\Ufst,max}\simeq \bigoplus_{i\in\Ic}\textup{End}_{\Rfut}(L_i^*)$$
with $L_i^*:=L_i\otimes_{\Rfut}\Rfutst$. 
\item[(viii)] The covering $\phi$ (as in (vi)) is not branched at points of $\Cf\cap C$. 
Let $\vb\in \Cf\cap C$ and let $\tilde{\vb}$ be a closed point lying above $\vb\in \Cf\cap C$.  
There exists a unique bijection $i\to \chi_i$ from $\Ic$ onto $\textup{DS}_{W_0r,\vb}$ such that 
$L_{i,\tilde{\vb}}:=L_i\otimes \Cb_{\tilde{\vb}}\simeq V_{\chi_{i,\vb}}$ as representation of $\Hc_\vb$, 
via the canonical isomorphism $\Hc_\vb\simeq \Hc_{\Rfut,\tilde{\vb}}$.
\end{enumerate} 
\end{prop}
\begin{proof}
Assertion (i): this is immediate from the definition. 

Assertion (ii): Since $\Sigma_{W_0r,C}^\Cf$ is central and is a $\Zc^\Uf$-eigenfunction with eigenvalue 
$W_0r$, it is clear that $M_{W_0r,C}^\Uf$ is finitely generated over $\Rfu$. It is also obviously torsion free over 
$\Rfu$, implying that the localization $M_{W_0r,C,Q}^\Uf$ is free over $\Rf_Q$ for all $Q\in\Uf$ (since the local rings  are 
principal ideal domains). Since $M_{W_0r,C}^\Uf$ is finitely generated, this implies that it is a locally free 
module.
 
Assertions (iii): We need to verify that the kernel of the surjective $\Hc_\Rfu$-module homomorphism 
$\Hc_\Rfu\ni h\to he\in M_{W_0r,C}^\Rfu$ is a two-sided ideal. This is a trivial consequence of the centrality of $e$.
This turns $M_{W_0r,C}^\Rfu$ into a $\Hc_\Rfu$-algebra.

Assertions (iv) and (v): 
We have already established that $M_{W_0r,C}^\Rfu$ is a locally free 
$\Hc_\Rfu$-algebra of finite rank.
When we specialize at $\vb\in C\cap \Uf$, the may use the fact \cite[Corollary 5.8]{DO} that 
$e_\vb\in \Sc_\vb\subset \Hc_\vb^*$ is a central idempotent of the Schwartz algebra $\Sc_\vb$ to see 
that the specialization $M_{W_0r,C,\vb}=M_{W_0r,C}^\Uf\otimes \Cb_\vb$ equals the 
finite dimensional semisimple algebra as stated in (iv).
Hence the trace form of $M_{W_0r,C}^\Uf$ is non degenerate over $\Kf$, and by shrinking  
$\Ufst$ sufficiently we may assume that the discriminant of the trace form of 
$M_{W_0r,C}^\Uf$ is non vanishing on $\Uf^*$. 
We arrive at the conclusion of (v) (see e.g. \cite[Section IV.3]{A}).

Assertion (vi): By (v) it follows that  $M_{W_0r,C}^\Kf:=M_{W_0r,C}^\Uf\otimes_\Rfu \Kf$ 
is a finite separable $\Kf$-algebra. Hence $M_{W_0r,C}^\Kf$ 
isomorphic to a finite direct sum of central simple algebras over finite separable field extensions $\Kf_i$ of $\Kf$
 (cf. \cite[Section IV.3]{A}). 
Tsen's Theorem \cite[Section III.15]{A} implies that in fact $M_{W_0r,C}^\Kf$ is 
a direct sum of full matrix algebras $\textup{Mat}_{d_i\times d_i}(\Kf_i)$.  Consider the compositum $\Lf$ of 
the $\Kf_i$. It is easy to see that for all $i$, the tensor product $\Kf_i\otimes_\Kf \Lf$ is isomorphic to a finite direct sum 
of copies of $\Lf$. It follows that $\Lf$ is a finite separable extension of $\Kf$ such that 
$M_{W_0r,C}^\Lf:=M_{W_0r,C}^\Uf\otimes_\Rfu \Lf=M_{W_0r,C}^\Kf\otimes_\Kf \Lf$ is isomorphic to a finite direct sum 
of full matrix algebras $M_{W_0r,C,i}^\Lf\simeq \textup{End}_\Lf(\Lf^{d_i})$  over $\Lf$.
Let $\Pi_i$ denote the projection homomorphism $M_{W_0r,C}^\Lf\to M_{W_0r,C,i}^\Lf$.
Let $\Rfut$ be the integral closure of $\Rfu$ in $\Lf$, which is a Dedekind domain
with fraction field $\Lf$. The inclusion $\Rfu\subset \Rfut$ gives rise to a ramified covering $\phi:\Uft\to\Uf$. We may 
assume, by shrinking $\Ufst$ (as in (v)) if necessary, that $\phi:\Uftst:=\phi^{-1}(\Ufst)\to\Ufst$ is finite \'etale.  

Now $\tilde{M}_{W_0r,C}^\Uf:=M_{W_0r,C}^\Uf\otimes_\Rfu\Rfut\subset M_{W_0r,C}^\Lf$ is clearly an $\Rfut$-order in 
$M_{W_0r,C}^\Lf$. Hence the projection $\Pi_i(\tilde{M}_{W_0r,C}^\Uf)$ is an $\Rfut$-order in $M_{W_0r,C,i}^\Lf\simeq 
\textup{End}_\Lf(\Lf^{d_i})$. By \cite[Theorem IV.4.1,Proposition IV.4.7]{A} this is contained in a maximal $\Rfut$-order 
in $M_{W_0r,C,i}^\Lf$ which must be a trivial Azumaya algebra of the form 
$\textup{End}_\Rfut(L_i)$
where $L_i\subset \Lf^{d_i}$ denotes an $\Rfut$-lattice. Hence we have 
$$
\tilde{M}_{W_0r,C}^\Uf\subset \bigoplus_i \Pi_i(\tilde{M}_{W_0r,C}^\Uf)\subset \bigoplus_i \textup{End}_\Rfut(L_i)=:\tilde{M}_{W_0r,C}^{\Uf,max}
$$
This turns $L_i$ into a $\Hc_\Rfut$-module, and proves (vi).

Assertions (vii) and (viii):
Localizing at $\Uftst$ we obtain $\tilde{M}_{W_0r,C}^\Ufst\subset \bigoplus_i \textup{End}_\Rfut(L_i^*)$, where 
$L_i^*:=L_i\otimes_\Rfut \Rfutst$ is a $\Rfutst$-lattice in $\Lf^{d_i}$. Since $\tilde{M}_{W_0r,C}^\Ufst$ is separable by (v), we have 
$\tilde{M}_{W_0r,C}^\Ufst=(\tilde{M}_{W_0r,C}^\Ufst)^*:=\textup{Hom}_{\Rfutst}(\tilde{M}_{W_0r,C}^\Ufst,\Rfutst)$. As in the proof of 
\cite[Corollary IV.4.5]{A} this implies that $\tilde{M}_{W_0r,C}^\Ufst$ is a maximal $\Rfutst$-order, hence we actually have:
$\tilde{M}_{W_0r,C}^\Ufst=\bigoplus_i \textup{End}_\Rfut(L_i^*)$, proving (vii). 
When we take $\tilde{\vb}$ lying above a $\vb\in \Ufst\cap C$, then specializing at $\tilde{\vb}$ and comparing (iv) and (vii) we obtain (viii). 
Observe that the structural result (vi) implies that $\phi$ can not be ramified at $\vb$. Indeed, when assuming the contrary then 
(vi) would imply that 
the number of discrete series characters of $\Hc_\vb$ with generic central 
character $W_0r$ would be strictly smaller than the number of such discrete series characters of $\Hc_{\vb'}$ with $\vb'\in C$ in a small 
neighborhood of $\vb$. This contradicts \cite[Corollary 5.11]{OpdSol2}. 

   
\end{proof}
\begin{prop}\label{prop:limlim} Let $b\in\Bc_{gm}$. 
The family $\Qc^{reg}_b\ni\vb\to \ep(b;\vb)\textup{Ind}_D(b;\vb)$ is continuous in the sense of \cite[Section 3]{OpdSol2}.
\end{prop}
\begin{proof}
The discrete series character $\ep(b;\vb)\textup{Ind}_D(b;\vb)$ of $\Hc_\vb$ has 
generic central character $gcc_\vb(\pi_\vb)=W_0r$ with $r=r_b$ by Corollary \ref{cor:gVcc}.     
By \cite[Definition 5.10]{OpdSol2} 
we know therefore that $\ep(b;\vb)\textup{Ind}_D(b;\vb)$ is also the specialization at $\vb$  
of a unique continuous family of generic irreducible discrete series 
characters $\chi\in\Delta_{W_0r_b}^{gen}(\Rc)$ with generic central character $W_0r_b$ 
on $\Qc_b^{gen}$.  By \cite[Corollary 5.8]{OpdSol2} $\Delta_{W_0r_b}(\Rc)$, is a constant sheaf 
with finite fiber. 
Let $C$ be a connected component of $\Qc^{reg}_b$, and let $\Cf$ be any 
irreducible real curve in $\Qc$ which intersects with $C$. Let $I_\Cf\subset \Cf\cap C$
be connected. By the above it is enough to show that $\ep(b;\vb)\textup{Ind}_D(b;\vb)$
and $\chi_\vb$ are equivalent when $\vb\in I_\Cf$.
In other words, it is enough to show that  
for all  $\vb\in I_\Cf$ we have $\lim_{\ep\to 0}[\chi_{\vb^\ep}]=b$. 
By Proposition \ref{prop:basicM}, the restriction of $\chi$ to $\Cf\cap C$ is 
realized by an algebraic model $L$ defined over some Zariski-open set $\Uft$
of a ramified cover $\tilde\Cf$ of $\Cf$, such that $I_\Cf\subset \Uf$ and there is no 
ramification at the points of $I_\Cf$. Choose a lift $\tilde I_\Cf$ of $I_\Cf$. 

By Corollary \ref{cor:gVcc}
we know that $b=\textup{Ind}_s(b_s)$, where $gcc_B(b)=W_0r$ for some generic residual point $r=sc$ 
with $s\in T_u$ a fixed special point, and $c=\exp(\xi)$ with $\xi$ a generic linear residual point for 
$\Hb(R_{s,1},T_s(E),F_{s,1};k_s)$. 
By Proposition \ref{prop:basicM} we know that $\ep(b,C)\textup{Ind}_D(b,C;\vb)$ is realized by an algebraic 
family $L$ defined over $\Rfutst$. By shrinking $\Ufst$ if necessary, we may assume that for all 
$\tilde\vb\in \Uftst$ we have, for all $r,r'\in W_0r$, that $sW_sc\not=s'W_{s'}c'$ (with $r'=s'c'$) if and only if 
$sW_sc(\vb)\cap s'W_{s'}c'(\vb)=\emptyset$ with $\vb=\phi(\tilde\vb)$.
Let $m_s$ be the ideal in $\Rfutst[X]$ corresponding to the finite set $sW_sc$ of $\Rfutst$-points
of $T$. By our choice of $\Rfutst$, we have that $m_s+m_s'=\Rfutst[X]$ whenever $sW_sc\not=s'W_{s'}c'$.
Hence by the Chinese remainder theorem we have 
$\widehat{\Rfutst[X]}_{W_0r}\simeq \oplus_{s'\in W_0s} \widehat{\Rfutst[X]}_{s'W_{s'}c}$.
Let $1=\sum_{s'\in W_0s}e_{s'}$ be the corresponding decomposition in orthogonal idempotents of 
$1\in \widehat{\Rfutst[X]}_{W_0r}$. Let $\Hc_{\Rfutst}=\Hc_\Lambda\otimes_\Lambda\Rfutst$.
Since $\Hc_\Rfutst$ contains the Bernstein subalgebra $\Ac_\Rfust\simeq \Rfust[X]$, we 
can define the idempotents $e_s\in\widehat{(\Hc_{\Rfutst})}_{W_0r}$. Since 
$gcc(L_b)=W_0r$, we can consider $L$ as module over  $\widehat{(\Hc_{\Rfutst})}_{W_0r}$.
Hence the $e_s$ act on $L$, and are mapped to idempotent elements in 
$\textup{End}_{\Rfutst}(L)$ such that $\textup{Id}=\sum_{s'\in W_0s}e_{s'}$. 

Following Lusztig \cite[Lemma 8.14, Lemma 8.15]{Lus2} we know that 
$\Hc_{s,\Rfutst}:=\widehat{\Hc_\Rfutst(\Rc_s)}_{W_sc}\rtimes \Gamma_s$
is isomorphic to the 
algebra 
$\widehat{(\Hc_{\Rfutst})}_{s,W_0r} :=e_s\widehat{(\Hc_{\Rfutst})}_{W_0r} e_s$, 
via the map (in the notations loc. cit.) $f.T_{w,s}.\gamma\to f.e_s.T_w.e_s.T_\gamma.e_s$.
Hence the finite, 
locally free $\Rfust$ module $L_s:=e_sL$ is a module over 
$\Hc_{s,\Rfutst}$, and in particular over the finite 
type Hecke subalgebra $\Hc_{s,\Rfutst,0}:=\Hc_\Rfutst(W(R_{1,s}))\rtimes \Gamma_s$
in the obvious way. 
By shrinking $\Rfutst$ if necessary we may assume that $L_s$ is actually 
a free $\Rfutst$-module. 

Let $\chi_s$ denote the character of $L_s$ (consided as a function of $\vb\in I_\Cf$).
By Corollary \ref{cor:lim} we see that for each $\vb\in I_\Cf$ we have 
$\lim_{\ep\to 0}[\chi_{\vb^\ep}]=\textup{Ind}_s(b_{\vb,s})$, where 
$b_{\vb,s}=\lim_{\ep\to 0}[\chi_{s,\vb^\ep}]$. This limit clearly only depends on 
the restriction of $L_s$ to $\Hc_{s,\Rfutst,0}$. 
Now it is well known 
that (see e.g. \cite[Theorem 5]{Opd0}) 
\footnote{For a discussion of the rationality properties of characters for finite Hecke algebras, see also \cite[section 9.3]{GP}, and the references therein.}
the irreducible characters 
of $\Hc_{s,\Rfutst,0}$ take values in $\Lambda$. 
This implies that $b_{\vb,s}$ is independent of $\vb \in I_\Cf$, as desired.
\end{proof}
\begin{cor}\label{c:loc-const}
$\ep(b,\cdot)$ is locally constant on $\Qc^{reg}_b$.
\end{cor}
\begin{proof}
The continuity of $\Qc_b^{reg}\ni\vb\to\ep(b;\vb)\textup{Ind}_D(b;\vb)$ implies that 
on each connected component $C$ of $\Qc_b^{reg}$ the sign $C\ni\vb\to\ep(b,\vb)$ is 
continuous, hence constant.
\end{proof}
\begin{defn} If $C\subset \Qc_b^{reg}$ is a connected component,
we define $\ep(b,C):=\ep(b,\vb)\in\{\pm 1\}$ for any choice of $\vb\in C$. 
\end{defn}
\subsection{Limits of discrete series}
Let $b\in\Bc_{gm}$  let $C\subset \Qc^{reg}_b$ be a connected component, and  
let $\vb_0\in\partial{C}$. 
Choose a connected real algebraic curve $\Cf_r=\Cf\cap\Qc\subset\Qc$ which meets $C$, and 
contains $\vb_0\in\Cf_r$ as a smooth point of $\Cf_r$.   
Extend $\Cf_r$ to a complex affine algebraic curve $\Cf\subset \Qc_\mathbb{C}$.
Assume we have chosen the structures as in Proposition \ref{prop:basicM} with 
respect to $W_0r=gcc_B(b)$, $\vb_0$, $C$ and $\Cf$. Let $\vb$ be a point in $\Cf\cap C$
such that $\vb_0$ and $\vb$ are in the same connected component of $\Cf\cap\overline{C}$. 
Given a $\tilde{\vb}\in \Uftst$
above $\vb\in\Cf\cap C$, let us denote by $L_b$ the unique $\Rfut$-lattice 
 as in Proposition \ref{prop:basicM}(vi) such that $L_{i,\tilde{\vb}}\simeq V_{\chi_{b,\vb}}$
 (according to Proposition \ref{prop:basicM}(viii)), where 
 $\chi_{b,\vb}=\ep(b;\vb)\textup{Ind}_D(b;\vb)$. The interval $I_\Cf\subset \Cf_r$ connecting 
 $\vb$ to $\vb_0$ has a unique lift $\tilde{I}_\Cf$ starting at $\tilde{\vb}\in\Uftst$. Let 
 $\tilde{\vb}_0\in\Uft$ be the unique endpoint of this lifted interval lying above $\vb_0$.  
\begin{defn}
We denote by $\textup{Ind}_D(b,C;\vb_0)$ the unique virtual character 
of $\Hc_{\vb_0}$ such that $\ep(b,C)\textup{Ind}_D(b,C;\vb_0)$ is realized 
by $L_{b,\tilde{\vb}_0}$. We call this representation the \emph{limit of the family of 
discrete series $\chi_b=\ep(b;\vb)\textup{Ind}_D(b;\vb)$ along $\Cf$ from $C$ at $\vb_0$}.
\end{defn}
\begin{rem}
A priori $\textup{Ind}_D(b,C;\vb_0)$ may depend on the chosen curve
and the model $L_b$ 
of  $\ep(b;\vb)\textup{Ind}_D(b)$ defined over $\Rfut$, 
but we suppress this from the notation. 
\end{rem}
\begin{cor}
$\ep(b,C)\textup{Ind}_D(b,C;\vb_0)$ is a genuine tempered character.
\end{cor}
\begin{proof}
Since the limit $\ep(b,C)\textup{Ind}_D(b,C;\vb_0)$ was defined by specializing the 
$\Rfut$-module $L_b$ at $\tilde{\vb}=\tilde{\vb}_0$, it is by definition a genuine character. 

The central character of $L_b(\tilde{\vb})$ is $W_0r(\vb)$ (with $\vb=\phi(\tilde{\vb})$), 
hence for all $z\in\Zc_{\Rfu}$ we have $z(W_0r)\in\Rfu$. Recall that $\Zc_\Rfu=\Ac_\Rfu^{W_0}$,  
hence the generalized $\Ac_\Rfu$ weight spaces belongs to the set 
$W_0r\subset T(\Rfu)$. 

Let $r'(\vb_0)\in W_0r(\vb_0)$, and let $\tilde{\vb}\in\tilde{I}_\Cf$ be close to $\tilde{\vb}_0$.
Choose a basis $\{x_1,\dots,x_n\}$ of $X$, and consider the corresponding commuting elements  
$\theta_i\in\Ac $ acting in $L_b(\tilde{\vb})$. 
For each $i$ consider the set $S_{i,r'(\vb_0)}:=\{r''\in W_0r\mid x_i(r''(\vb_0))=x_i(r'(\vb_0))\}$ 
of generalized eigenvalues of $\theta_i$ (acting in  $L_b(\tilde{\vb})$) which coalesce to 
$x_i(r'(\vb_0))$ at $\vb=\vb_0$. Then the projection 
$\Pi_{r',\tilde{\vb}}$ onto the direct sum of the generalized $\Ac$-weight spaces  
in $L_b(\tilde{\vb})$ which coalesce to $r'(\vb_0)$ is the composition of the commuting projection 
operators 
$\Pi_{i,r',\tilde{\vb}}$ onto the direct sum of the generalized eigenspaces $r''(\vb)$ of  
$\theta_i(\tilde{\vb})$ with $r''\in S_{i,r'(\vb_0)}$. 
By holomorphic functional calculus we have 
$$
\Pi_{i,r',\tilde{\vb}}=\frac{1}{2\pi i}\int_{\partial{D_i}}(z\textup{Id}-\theta_i(\tilde{\vb}))^{-1}dz
$$
where $D_i$ is a fixed disk around $x_i(r'(\vb_0))$ such that $x_i(r''(\vb))\in D_i$ if and 
only if $r''\in S_{i,r'(\vb_0)}$ (such $D_i$ clearly exist, provided $\vb$ is sufficiently close to $\vb_0$).
Hence the $\Pi_{i,r',\tilde{\vb}}$ are continuous in $\tilde{\vb}$. In particular $\Pi_{r',\tilde{\vb}}$
is continuous in $\tilde{\vb}$, implying that the dimension of its image is independent of $\vb$.
Therefore $r'(\vb_0)$ is a generalized $\Ac$-weight of $L_b(\tilde\vb_0)$ if and only 
if for some $\vb\in I_\Cf$ sufficiently close to $\vb_0$, there exists a $r''(\vb)$ 
with $r''(\vb_0)=r'(\vb_0)$ (i.e. $r''\in\cap_i S_{i,r'(\vb_0)}$) such that $r''(\vb)$ 
is a generalized $\Ac$-weight of the discrete series 
representation $L_b(\tilde\vb)$. In particular, $r'(\vb)$ is a limit of 
$\Ac$-weights which meet the Casselman condition  
for temperedness  \cite[Lemma 2.20]{Opd1}. This is a closed condition, 
hence the generalized $\Ac$-weights of $L_b(\tilde\vb_0)$ 
satisfy the Casselman conditions themselves. The lemma loc. cit. implies that 
$L_b(\tilde\vb_0)=\ep(b,C)\textup{Ind}_D(b,C;\vb_0)$ is tempered.
\end{proof}
\begin{cor}
The tempered character $\ep(b,C)\textup{Ind}_D(b,C;\vb_0)$ is a member of an algebraic family 
of characters of $\Hc_\Rfut$-characters with values in $\Rfut$. For generic element $\tilde{\vb}\in\Uft$
this character is irreducible, and for  $\tilde{\vb}\in\Uftst\cap C$ it is an irreducible discrete series.
\end{cor}
\begin{rem}
The limit $\ep(b,C)\textup{Ind}_D(b,C;\vb_0)$ may be irreducible or not. In the case of 
the Hecke algebras of type $\textup{C}_n^{(1)}$ the limits of discrete series were constructed 
using the geometric model \cite{K} of characters of the generic affine Hecke algebra of this type. 
In this situation it is known that the limits of the discrete series 
at nontrivial singular parameters are always irreducible \cite{CKK}. In general the limit 
to the trivial parameter $\vb=1$ will be a reducible character of the affine Weyl group
(the only exceptions being the "one $W$-type" discrete series characters).
\end{rem}
\begin{prop}\label{prop:limlim0}
We have $\lim_{\ep\to 0} [\textup{Ind}_D(b,C;\vb_0^\ep)]=b\in \Bc_{gm}\subset \overline{\mc{R}}_\Zb(W)$
\end{prop}
\begin{proof}
The  proof is exactly the same as the proof of Proposition \ref{prop:limlim}.
\end{proof}
%
\begin{cor}
The covering $\phi$ is not ramified at the points of $\Cf\cap\overline{C}$.
\end{cor}
\begin{proof}
The argument of Proposition \ref{prop:limlim} proves that 
$b_s=\lim_{\ep\to 0}[\chi_{s,\vb^\ep}]$ for any $\vb$ in a neighborhood of $\vb_0$.
On the other hand, the discrete series of $\Hc_\vb$ are parameterized by the 
$b\in\Bc_{\vb-gm}=\Bc_{gm}\cap \Yc_{\vb-gm}$ by Corollary \ref{cor:lim}, Proposition \ref{prop:Bgm} 
and Subsection \ref{subsec:dirac}. Moreover, for $b$ in the image of $\textup{Ind}_s$, 
the limit map $b\to b_s$ is an isometry by Corollary \ref{cor:lim}.  
Hence the generic discrete series with generic central character $W_0r$ 
are also parameterized by the corresponding elliptic class $b_s$ of $W_s$.
Therefore the algebraic continuation of the discrete series characters $\ep(b,C)\textup{Ind}_D(b;\vb)$ 
(realized by $L_b$) 
on $\Uftst$ can not have monodromy around $\vb_0$, by the above. 
We conclude that the character values are 
regular functions in a neighborhood of $\vb_0\in\Uf$.
\end{proof}
\subsection{The rationality of the generic formal degree}
\subsubsection{The universality of the rational factors of the generic formal degree}
The following result follows now simply from \cite[Corollary 5.7]{CiuOpd1}, \cite{COT} and \cite{OpdSol2}:
\begin{thm}\label{thm:thm:fd}
Fix $b\in \Bc_{gm}$.
The formal degree of a continuous family 
of virtual discrete series characters $\Qc_b^{reg}\ni\vb\to\textup{Ind}_D(b;\vb)$ as a function of $\vb\in\Qc_b^{reg}$ 
is a rational function of the form $d_bm_b$, with $d_b\in\Qb^\times$.
\end{thm}
\begin{proof}
By \cite[Corollary 5.7]{CiuOpd1} the formal degree of $\textup{Ind}_D(b;\vb)$ is a linear combination of 
rational functions of the $\vb_s$, which only only depends on 
the elliptic class $\lim_{\ep\to 0}\textup{Ind}_D(b;\vb^\ep)$. 
This is, by definition, $b$ (and therefore 
independent of $\vb$). On the other hand we have shown that this family is continuous and 
(by \cite{COT})  has generic central character $W_0r_b$. By \cite[Theorem 4.6]{OpdSol2}
we conclude that the formal degree has the form $d_b(C)m_b$ for some constants $d_b(C)$
depending on the connected component $C$ of $\Qc_b^{reg}$ in which $\vb$ lies.
But the rationality implies that these constants need to be equal on all chambers.  
\end{proof}

\begin{cor}\label{cor:unique}
We can choose the basis vector $b\in\Bc_{gm}$ uniquely such that $d_b>0$.
With this choice, $\ep(b,C)$ will be equal to the sign of $m_b(\vb)$ for $\vb\in C$. 
\end{cor}
\subsection{Proof of Theorem \ref{thm:uniform}} 
The proof of Theorem 1.1 is all completed now, and we point out where the various parts of it have been proved. Part (a) is in Proposition \ref{prop:Bgm} and Corollary \ref{cor:unique}. Part (b) is in subsection \ref{subsec:dirac}. Part (c) is Corollary \ref{cor:gVcc}. Part (d) is in Proposition \ref{prop:limlim}. Part (e) is Corollary \ref{c:loc-const}. Part (f) follows from Corollary \ref{cor:lim}, Proposition \ref{prop:Bgm}, and subsection \ref{subsec:dirac}. Part (g) is Theorem \ref{thm:thm:fd}.
\section{Explicit results; comparison with Kazhdan-Lusztig-Langlands classification}
In this section we will compare the uniform classification of the discrete series 
with the Kazhdan-Lusztig-Langlands classification when an affine Hecke algebra  
arises in the context of a unipotent type of an unramified 
simple group defined over a non-archimedean local field. We will also compute, for all non-simply laced affine Hecke 
algebras  of simple type, the canonically positive basis $\Bc_{gm}$ and the fundamental rational 
constants $d_b$ appearing in the generic Plancherel measure. In addition, we will show that 
in each connected component of $\mathcal{Q}^{reg}_{b}$ the generic discrete series 
character $\epsilon(b;\vb)\textup{Ind}_D(b;\vb)$ takes values in $Q_\Lambda$, 
the quotient field of $\Lambda$. 
Needless to say, the results in this section are 
based on case-by-case methods.
\subsection{Determination of the canonically positive basis $\Bc_{gm}$}
\subsubsection{The sign of $m_b(\vb)$} In light of Corollary \ref{cor:unique}, we need to analyze the sign of $m_b(\vb)$.   Let $r=r_b$ be a residual point. As in \cite[(40)]{OpdSol2}, we write 
\begin{equation}\label{e:m_b}
m_{W_0r}=\frac{\prod'_{\alpha\in R_1}(\alpha(r)^{-1}-1)}{\prod'_{\alpha\in R_1}(v_{\alpha^\vee}^{-1} \alpha(r)^{-1/2}+1)\prod'_{\alpha\in R_1}(v_{\alpha^\vee}^{-1}v_{2\alpha^\vee}^{-2}\alpha(r)^{-1/2}-1)},
\end{equation}
with the notation as in {\it loc.~cit.}. In particular, recall that $R_1$ is the reduced part of $R_0\cup\{2\alpha|\alpha^\vee\in 2Y\cap R_0^\vee\}$ and the parameters $v_{\beta^\vee}$ are defined in terms of the $v(\bs)$ as in \cite[(7),(8)]{OpdSol2}. Here $\prod'$ means that the identically zero factors are ignored. When $\alpha\in R_0\cap R_1$, $v_{2\alpha^\vee}=1$ and each factor in the denominator simplifies to $(v_{\alpha^\vee}^{-2}\alpha(r)^{-1}-1)$. Moreover, the expression in (\ref{e:m_b}) is independent of the choice of representative $r$ in its $W_0$-orbit.

Specialize
\[\vb(\bs)=\vb^{f_s} ,\text{ with }\vb>1,\ f_s\in\mathbb R.
\]
Write $r=sc$ as before and let $\Hb(R_{s,1},T_s(T),F_{s,1};k_s)$ denote the corresponding graded affine Hecke algebra. Recall that the parameter function $k_s$ is given by \cite[(26)]{OpdSol2}
\begin{equation}
k_s(\alpha)=
\begin{cases}
\log_{\vb}(v_{\alpha^\vee}^2), &\text{ if } \alpha\in R_0\cap R_1, \text{ or if }\alpha=2\beta,\ \beta\in R_0 \text{ and }\beta(s)=1,\\
\log_{\vb}(v_{\alpha^\vee}^2 v_{2\alpha^\vee}^4), &\text{ if }\alpha=2\beta,\ \beta\in R_0,\ \beta(s)=-1.\\
\end{cases}
\end{equation}
Write $c=\vb^{\bar c}.$

\begin{prop}\label{p:m-bar}
The sign of $m_{W_0,r}$ from (\ref{e:m_b}) equals the sign of the expression
\begin{equation}\label{e:m-bar}
\bar m_{W_0\bar c} =\frac{\prod'_{\alpha\in R_{s,1}}\alpha(\bar c)}{\prod'_{\alpha\in R_{s,1}}(\alpha(\bar c)-k_s(\alpha))}.
\end{equation}
\end{prop}

\begin{proof} For every $\alpha\in R_1$, $\alpha(s)$ is a root of unity. Write $R_1=R_{s,1}\sqcup R_1^{s,-1}\sqcup R_1^{s,z}$, where $R_1^{s,-1}=\{\alpha\in R_1 |\alpha(s)=-1\}$ and $R_1^{s,z}=\{\alpha\in R_1 |\alpha(s)^2\neq 1\}.$ If $\alpha$ is a  root in $R_1\setminus R_0$, then  $\alpha$ is a root of type $A_1$ in $B_n$ and from the classification of isolated semisimple elements $s$ in this situation, we see that $\alpha\in R_{s,1}.$

We break up the product expression for $m_{W_0r}$ according to the three types of roots: $R_{s,1}$, $R_1^{s,-1}$, $R_1^{s,z}$. 
Take $R_1^{s,-1}$ first and we analyze the contribution of its numerator and denominator. The numerator is a product of expressions $(-\alpha(c)-1)(-\alpha(c)^{-1}-1)>0$, since $\alpha(c)>0$, one factor for each positive root $\alpha\in R_1^{s,-1}$. For the denominator, for each positive root $\alpha\in R_1^{s,-1}$ (which by the remark above it is in $R_0$), we have a factor $(-v_{\alpha^\vee}^{-2}\alpha(c)^{-1}-1)(-v_{\alpha^\vee}^{-2}\alpha(c)-1)>0$. Therefore, the part of the expression corresponding to $R_1^{s,-1}$ is positive.

 Consider now $R_1^{s,z}.$ Its contribution equals 
$\frac
{\prod'_{\alpha\in R_1^{s,z}}(\alpha(s)\alpha(c)-1)}
{\prod'_{\alpha\in R_1^{s,z}}(\alpha(s)v_{\alpha^\vee}^{-2} \alpha(c)-1)}.$
The argument in \cite[Theorem 3.27(v)]{Opd1} shows that both the numerator and the denominator are polynomial expressions in $\vb$ with rational (in fact integer) coefficients. Moreover, since $\alpha(s)$ is not real, these polynomials do not afford real roots in $\vb$. It follows that they are always positive or always negative for real $\vb$. But it is clear that for $\vb=0$ the fraction above equals $1>0$, and therefore it is always positive. 

In conclusion, the sign of $m_{W_0r}$ equals the sign of the expression $$m_{W_0r}=\frac{\prod'_{\alpha\in R_{s,1}}(\alpha(c)^{-1}-1)}{\prod'_{\alpha\in R_{s,1}}(v_{\alpha^\vee}^{-1} \alpha(r)^{-1/2}+1)\prod'_{\alpha\in R_{s,1}}(v_{\alpha^\vee}^{-1}v_{2\alpha^\vee}^{-2}\alpha(r)^{-1/2}-1)}.$$
The numerator can be rewritten as $\prod_{\alpha\in R_{s,1}^+}\alpha(c)^{-1} (-1) (\alpha(c)-1)^2$ and therefore its sign is $(-1)^{n_s}$, where $n_s=\#\{\alpha\in R_{s,1}^+ | \alpha(c)\neq 1\}=\#\{\alpha\in R_{s,1}^+ | \alpha(\bar c)\neq 0\}$.  But this is also the sign of the numerator in (\ref{e:m-bar}).

Let $\alpha\in R_0^+\cap R_{s,1}$ be given. Then $\alpha$ and $-\alpha$ give a contribution in the denominator of $v_{\alpha^\vee}^{-4}(\alpha(c)^{-1}-v_{\alpha^\vee}^2)(\alpha(c)-v_{\alpha^\vee}^2)=v_{\alpha^\vee}^{-4}(\vb^{-\alpha(\bar c)}-\vb^{k_s(\alpha)})(\vb^{\alpha(\bar c)}-\vb^{k_s(\alpha)})$. Clearly, the sign of this expression is the same as the sign of $(-\alpha(\bar c)-k_s(\alpha))(\alpha(\bar c)-k_s(\alpha))$. 

Now suppose that $\alpha=2\beta$ with $\beta\in R_0$. If $\beta(s)=1$, the factors of the form $(v_{\alpha^\vee}^{-1}\alpha(r)^{-1/2}+1)=(v_{\alpha^\vee}\beta(c)^{-1}+1)$ are all positive and can be ignored. So the contribution in the denominator comes from the factors $(v_{\alpha^\vee}^{-1}v_{2\alpha^\vee}^{-2}\alpha(r)^{-1/2}-1)=(\vb^{-k_s(\alpha)/2}\beta(c)^{-1}-1).$ Grouping together the factors corresponding to $\alpha$ and $-\alpha$ as before and multiplying by powers of $\vb$, we see that the contribution to the sign is given by $(\vb^{\alpha(\bar c)/2}-\vb^{k_s(\alpha)/2})(\vb^{-\alpha(\bar c)/2}-\vb^{k_s(\alpha)/2})$. But this has the same sign as $(-\alpha(\bar c)-k_s(\alpha))(\alpha(\bar c)-k_s(\alpha))$. 

Finally, if $\beta(s)=-1$, then the  factor $(v_{\alpha^\vee}^{-1}v_{2\alpha^\vee}^{-2}\alpha(r)^{-1/2}-1)=-(v_{\alpha^\vee}^{-1}v_{2\alpha^\vee}^{-2}\beta(c)^{-1}+1)$ is negative and the contributions of these factors for $\alpha$ and $-\alpha$ cancel out. Thus we remain with the factor $(v_{\alpha^\vee}^{-1}\alpha(r)^{-1/2}+1)=(-\vb^{-k_s(\alpha)/2} \beta(c)+1)$ and the same analysis as in the previous case applies.

\end{proof}

\begin{rem}\label{r:sign-spec} Fix $b\in \Bc_{gm}$ and suppose $\vb_0\in \mathcal Q_b^{reg}\setminus \mathcal Q_b^{gen}$  is a nongeneric, but regular parameter. By definition, $b$ is in $\Bc_{\vb_0-m}$ and let $r_0=\lim_{\vb\to \vb_0}r_b(\vb).$   The  irreducible discrete series at $\vb_0$ in the family defined by $b$ is $\textup{ds}(b,\vb_0)=\epsilon(b;\vb_0)\textup{Ind}_D(b;\vb_0)$. By Theorem \ref{thm:uniform}, the sign $\epsilon(b;\vb)$ is locally constant and therefore taking $\lim_{\vb\to \vb_0}$ in the formula of Theorem \ref{thm:uniform}(g), we see that $\epsilon(b;\vb_0)$ equals the sign of $\lim_{\vb\to \vb_0} m_{W_0r_b(\vb)},$
where $m_{W_0r_b(\vb)}$ is defined by (\ref{e:m_b}).
\end{rem}

\subsection{Tables}
We present the explicit form of the canonically positive basis $\Bc_{gm}$, as well as the generic residual central characters $W_0r_b$, the constants $d_b$ in the case when $\Rc$ is a simply-connected, simple, root datum of type $C_n$ (with three parameters), $G_2$ and $F_4$ (with two parameters), and $E_6$, $E_7$, $E_8.$ For the affine Hecke algebra of a simple root datum of arbitrary isogeny, one may deduce the relevant information from the cases listed above, by specializing the parameters appropriately and by the use of induction and restriction. See \cite{OpdSol2}, \cite{Opd4}, \cite{Slo}, \cite{Ree4}.

\subsubsection{$C_n^{(1)}$} Let $X=\mathbb Z^n=\langle\ep_1,\dots,\ep_n\rangle$, $R_0=\{\pm\ep_i\pm\ep_j,~i\neq j,~\pm\ep_i\}$, $F_0=\{\ep_i-\ep_{i+1},\ep_n\}$. For simplicity of notation, we use the coordinates $\ep_i$ for the basis of $Y$ as well. We have $W=X\rtimes W_0=\mathbb Z R_0\rtimes W_0.$ The affine simple roots are $F=\{(\ep_i-\ep_{i+1},0\}\cup \{(2\ep_n,0)\}\cup (-2\ep_1,1)$ and the affine Dynkin diagram (with parameters) is 
\begin{equation}
\xymatrix{v_0&v_1\ar@{<=}[l]\ar@{-}[r]&v_1\ar@{-}[r]&\dotsb&v_1\ar@{-}[l]&v_2\ar@{=>}[l]},
\end{equation}
Here the affine simple root is $\alpha_0=(-2\ep_1,1)$ and it gets the label $v_0.$ Let $\Hc(v_0,v_1,v_2)$ be the affine Hecke algebra with generators $N_0,N_1,\dots,N_n$. We need to switch to the Bernstein presentation. The algebra $\Hc(v_0,v_1,v_2)$ is generated $N_1,\dots,N_n$ and $\theta_1^{\pm 1},\dots,\theta_n^{\pm 1}$ (here $\theta_i=\theta_{\ep_i}$) subject to the relations
\begin{equation}
\begin{aligned}
&\theta_i N_i-N_i \theta_{i+1}=(v_1-v_1^{-1})\theta_i,\quad 1\le i\le n-1;\\
&\theta_i N_j=N_j \theta_i,\quad |i-j|\ge 2;\\
&\theta_nN_n-N_n\theta_n^{-1}=(v_1-v_2^{-1})\theta_n+(v_0-v_0^{-1}),
\end{aligned}
\end{equation}
the usual Hecke relations and the commutation of the $\theta_i$'s. In the change of presentation, we set $\theta_i=N_0 N_{s_{\ep_i}}$, where $s_{\ep_i}\in W_0$ is the reflection corresponding to the root $\ep_i.$ 

It it immediate that the assignment $N_0\mapsto -v_0^{-1}$ gives a surjective algebra homomorphism onto the finite Hecke algebra $\Hc_f(C_n,v_1,v_2)$ of type $C_n$ with parameters $v_1,v_2.$ Translating to the Bernstein presentation, we find that the assignment
\begin{equation}
\begin{aligned}
&N_i\mapsto N_i,&\theta_i\mapsto -v_0^{-1} N_{s_{\ep_i}}\quad 1\le i\le n,\\
\end{aligned}
\end{equation}
extends to a surjective algebra homomorphism onto the Hecke algebra $\Hc_f(C_n,v_1,v_2)$ of finite type. This allows us to lift every simple $\Hc_f(C_n,v_1,v_2)$-module to a simple module of the affine Hecke algebra. The simple $\Hc_f(C_n,v_1,v_2)$-modules are parameterized by bipartitions $(\lambda,\mu)$ of $n$. It is particularly useful to use Hoefsmit's construction of such modules, see \cite[pages 322--325]{GP}, since in that realization the $N_{s_{\ep_i}}$ act diagonally. 

We recall the construction. Denote by $V_{(\lambda,\mu)}(v)$ the simple module of $\Hc_f(C_n)$ parameterized by $(\lambda,\mu)$. Its basis is indexed by left-justified decreasing standard tableaux of shape $(\lambda,\mu).$ Let $\mathcal y$ denote such a Young tableau. Then
\begin{equation*}
\begin{aligned}
N_{s_{\ep_j}}\cdot \mathcal y=\mbox{ct}(\mathcal y,n-j+1)\mathcal y, \text{where }
\mbox{ct}(\mathcal y,k)=\begin{cases}v_1^{2(y-x)} v_2,&\text{ if $k$ occurs in }\lambda,\\ -v_1^{2(y-x)} v_2^{-1},&\text{ if $k$ occurs in }\mu,
\end{cases}
\end{aligned}
\end{equation*}
where $(x,y)$ are the coordinates of the box in which $k$ occurs. (The coordinates $(x,y)$ of a box in the Young tableau increase to the right in $y$ and down in $x$.)

This means that in $\widetilde V_{(\lambda,\mu)}(\vb)$, the lift of $V_{(\lambda,\mu)}$ to the affine Hecke algebra, we have
\begin{equation}\label{e:theta-action}
\theta_j \cdot \mathcal y=-v_0^{-1} \mbox{ct}(\mathcal Y, n-j+1)\mathcal y.
\end{equation}
The condition that a simple module is a discrete series module is that the eigenvalues of the product $\theta_1\cdot\theta_2\dotsb\theta_j$ are all smaller than $1$ in absolute value, for all $j=1,\dots,n.$ From (\ref{e:theta-action}), we see that when the absolute value of the specialization of $v_0$ is much larger than that of $v_1$ and $v_2$, then every $V_{(\lambda,\mu)}(v)$ is a discrete series modules. Moreover, the central characters of these modules are all distinct at generic values of the parameters, and since by \cite{OpdSol2}, the dimension of the space $\overline \Rc_\Zb(W)$ equals the number of bipartitions of $n$, it follows that the set $\{\lim_{v\to 1}\widetilde V_{(\lambda,\mu)}(v)\}$ is an orthonormal basis of $\Yc_{gm}=\overline \Rc_\Zb(W)$. In order to determine the canonically positive basis $\Bc_{gm}$, it remains therefore to determine the signs $\ep(b,\vb)$ for each $b=\lim_{v\to 1}\widetilde V_{(\lambda,\mu)}(v).$

To this end, we need to examine the formal degrees of these modules. Denote $R_0^{sh}=\{\pm \ep_i\}$, the short roots, and $R_0^{lo}=\{\pm\ep_i\pm\ep_j\}$, the long roots. Specialize
\[v_1=v,\quad v_2=v^{m_++m_-},\quad v_0=v^{m_+-m_-},
\]
where $v>1$ and $m_-,m_+\in \mathbb R$.
The formula for the generic formal degree \cite[Theorem 4.6 and (40)]{OpdSol2} gives in our particular case that the formal degree of a discrete series $\pi(b,v)$ with central character $W_0 r_b(\vb)$ equals:
\begin{equation}\label{e:fdeg-C}
\begin{aligned}
\textup{fdeg}(\pi(b,\vb))=&\frac{d_b ~\ep(b,m_+,m_-)  \prod'_{\alpha\in R_0}(\alpha(r_b(v))^{-1}-1)}
{ \prod'_{\alpha\in R_0^{lo}} (v^{-2}\alpha(r_b(v))^{-1}-1) }\\
&\cdot\frac 1{\prod'_{\alpha\in R_0^{sh}} (v^{-2m_+} \alpha(r_b(v))^{-1} -1)\prod'_{\alpha\in R_0^{sh}} (v^{-2m_-} \alpha(r_b(v))^{-1} +1)},
\end{aligned}
\end{equation}
where $\ep(b,m_+,m_-)$ is a sign to be explicitated. 
From \cite[Theorem 4.7]{CKK} (see the remark in the proof of  \cite[Theorem 4.12]{Opd4}), we know that \begin{equation}d_b=1.\end{equation}
Fix a bipartition $(\lambda, \mu).$ The corresponding central character $W_0r_b(\vb)$ is the $W_0$-orbit of the string
\[ r_{(\lambda,\mu)}=\left((-v^{2(y-x)}v^{2m_-}| (x,y)\in\lambda) ; (v^{2(y'-x')}v^{-2m_+}| (x',y')\in\mu)\right).
\]
and therefore
\[\bar c_{(\lambda,\mu)}=(\bar c_\lambda(m_-), \bar c_\mu(-m_+)), \text{ where }\bar c_\lambda(m)=\left((y-x) +m| (x,y)\in\lambda\right).
\]
From Proposition \ref{p:m-bar}, $\ep(b,m_+,m_-)=\ep(\lambda,m_-)\ep(\mu,-m_+)$ where $\ep(\lambda,m)$ (and similarly for $\mu)$ equals the sign of the expression
\begin{equation}
\frac
{\prod'_{\alpha\in B_{|\lambda|}}\alpha(\bar c_{\lambda}(m))}
{\prod'_{\alpha\in D_{|\lambda|}}(\alpha(\bar c_{\lambda}(m))-1)\prod'_{\alpha\in A_1^{|\lambda |}}(\alpha(\bar c_{\lambda}(m))-m)} .
\end{equation}

 As noticed before, in the chamber where $m_->>m_+>>0$, the discrete series are $V_{(\lambda,\mu)}(\vb).$ Let $\epsilon(\lambda,\mu)=\lim_{m_\pm\to\infty}\ep(b,m_+,m_-).$ In conclusion, we have proved:

\begin{prop}\label{p:B-gm-C}
For the affine Hecke algebra of type $C_n^{(1)}$, the canonically positive basis $\Bc_{gm}$ of $\Yc_{gm}=\overline \Rc_{\mathbb Z}(W)$ is $$\Bc_{gm}=\{\ep(\lambda,\mu) \widetilde V_{(\lambda,\mu)}(1): ~ (\lambda,\mu) \text{ bipartition of }n\},$$ where  $\widetilde V_{(\lambda,\mu)}(1)$ is the irreducible $W$-representation with central character $W_0s_{(\lambda,\mu)}$, $s_{(\lambda,\mu)}=(\underbrace{-1,\dots,-1}_{|\lambda|}, \underbrace{1,\dots,1}_{|\mu|})$, and whose restriction to $W_0$ is the irreducible $W_0$-representation labelled by the bipartition $(\lambda,\mu).$
\end{prop}

\subsubsection{$G_2$, $F_4$}\label{subsec:G2} For the affine Hecke algebras of types $G_2$ and $F_4$, we  work with the following coordinates. For $G_2$, the affine diagram is 
\[\xymatrix{\alpha_0\ar@{-}[r]&\alpha_1\ar@{-}[r]&\alpha_2\ar@{<=}[l]},
\]
with parameters $v(\bs_i)=v^{2k_i}$, $i=1,2$, while  
for $F_4$, it is
\[\xymatrix{\alpha_0\ar@{-}[r]&\alpha_1\ar@{-}[r]&\alpha_2&\alpha_3\ar@{<=}[l]&\alpha_4\ar@{-}[l]},
\]
with parameters $v(\bs_1)=v^{2k_1}$, $v(\bs_3)=v^{2k_2}$. In both cases, $v>1$ and $k_i\in\mathbb R.$

Let $\omega_i^\vee$ denote the fundamental coweights. A central character is of the form $r_b=s c$, where $c=v^{\sum_i a_i\om_i^\vee}$, where $s$ is an compact element of the torus.
The generic residual central characters are listed in Tables \ref{t:G2} and \ref{t:F4}. In this table, if $s\neq 1$, we specify it by the type of its centralizer in the second column. In the  third column, we give $c$ in the form $[a_i].$

 For each generic residual central character $W_0r_b$, we compute the function $m_b(\vb)$ using formula \cite[(40)]{OpdSol2}. To obtain the constant $d_b$, we compute the limits of $m_b(\vb)$ in equal parameter case, e.g., $k_1\to 1, k_2\to 1$ and also in the unequal parameter cases that appear in $E_8$. Then we compare the results with the formulas for formal degrees in \cite{Ree1} and \cite{Ree2}.  To complete the determination of the Langlands parameter, i.e., the representation of the component group, we computed, when needed, the $W_0$-structure of the specialization of the family of discrete series and compared it with the $K$-structure of the representations in Reeder's tables.
 The relevant unequal parameter cases that appear in $E_8$ are: the affine Hecke algebra of type $F_4$ that controls the subcategory of unipotent representations where the parahoric subgroup is $D_4$ in $E_8$, and the affine Hecke algebra of type $G_2$ for the subcategories of unipotent representations where the parahoric subgroup is $E_6$ in $E_8$.  For this comparison, we need to multiply the specialization of the formal degree in the unequal parameters Hecke algebra by the factor $\frac {\rho(1)}{P_J(v^2)}$, where $\rho$ is the appropriate cuspidal unipotent representation (whose dimension is given in the tables of \cite{C}) and $P_J(v^2)$ is the Poincar\'e polynomial of the finite Hecke algebra corresponding to the parahoric $J$.
 
In addition, the Iwahori-Hecke algebra for the quasisplit exceptional $p$-adic group ${}^3E_6$ (respectively, ${}^2E_7$) is isomorphic to a direct sum of three (respectively, two) copies  of an affine Hecke algebra with unequal parameters of type $G_2$ (respectively, $F_4$).  To compare  the formal degrees in these cases against the results of \cite{Ree2}, we need to multiply the specialization of formal degrees for the affine Hecke algebra by the factor
\[\frac {(q^{1/2}-q^{-1/2})^{n+1}}{|\Omega | \prod_{O} (q^{|O|/2}-q^{-|O|/2})},
\]
 where $n+1$ is the number of nodes in the affine Dynkin diagram, and $O$ ranges over the Galois orbits in the affine Dynkin diagram. See \cite[(25)]{Opd4} for more details. For the nonsplit inner forms, this procedure allows us to find the L-packet to which the representation should belong, but it is not sufficient in order to attach the representation of the component group. \footnote{This is a subtle question, see \cite{Opd4} for results in this sense. We plan to return to this question in future work.}



\subsection{Relation with Kazhdan-Lusztig parameters}\label{s:KL} Let $\Rc$ be a semisimple root datum. 
Let $G$ be the connected complex semisimple group with root datum $\Rc.$ Consider the generic affine Hecke algebra 
$\Hc$ with root datum $\Rc$ and equal parameters, i.e., $v(\bs)=v(\bs')=v$ for all $\bs,\bs'\in S.$ 
If $G$ is simply connected, the Kazhdan-Lusztig classification \cite{KL} applies to give a parameterization of the simple discrete series 
$\Hc_{v_0}$-modules, $v_0>1$, in terms of
\begin{equation}\label{e:KL}
\textup{DS}_{KL}(\Rc)=G\backslash \{(x,\phi)\mid x\in G\text{ elliptic},\ \phi\in \widehat{A(x)} \text{ such that } H^{top}(\mathcal B_x)^\phi\neq 0\}.
\end{equation}
This result has been extended by Reeder \cite{Ree4} (also see \cite{ABPS}) to the case where $\Rc$ has arbitrary 
isogeny type.
Recall that we say that $x$ is elliptic in $G$ if the conjugacy class of $x$ does not meet any proper Levi subgroup of $G$. Here we denoted by $A(x)=Z_G(x)/Z_G(x)^0Z(G)$ the component group of the centralizer of $x$ in $G$ (mod the center of $G$), by $\mathcal B_x$ the Springer fibre of $x$ in $G$, and by $H^{top}(\mathcal B_x)^\phi$ the $\phi$-isotypic component of $A(x)$ in the top cohomology of $\mathcal B_x$.

As it is well known, by the construction of Springer extended by Lusztig and Kato 
in the simply connected case (see for example \cite[section 8]{Ree1}), and further extended by Reeder in \cite[Section 3]{Ree4}
to the arbitrary semisimple case, that  
the full cohomology groups $H^\bullet(\mathcal B_x)^\phi$ carry an action of $W$. We define
\begin{equation}
\Bc_{DS}^{KL}=\{h_{x,\phi}:=H^\bullet(\mathcal B_x)^\phi\otimes\varepsilon\mid (x,\phi)\in \textup{DS}_{KL}(\Rc)\},
\end{equation}
where $\varepsilon$ is the sign character of $W$. 

Let $x$ be an elliptic parameter as above and write $x=su$ for the Jordan decomposition, with $s\in T_u.$ Let $\psi: SL(2,\mathbb C)\to G$ be the Lie homomorphism such that $\phi(\left(\begin{matrix}1&1\\0&1\end{matrix}\right))=u.$ Set $\tau=s\phi(\left(\begin{matrix} v^{-1} &0\\0 &v\end{matrix}\right))\in T_u T_v.$ Set $q=v^2$. If $\Rc$ is of simply connected type then formal degree is given by the formula (\cite{Opd4,Ree1}):
\begin{equation}\label{e:reeder}
\textup{fdeg}(\pi_{v,b_{x,\phi}})=\frac {\phi(1)}{|A(x)||Z(G)|} m_v(\tau), \text{ where }m_v(\tau)=q^{|R|/2}\frac{\prod'_{\alpha\in R}(\alpha(\tau)-1)}{\prod'_{\alpha\in R}(q\alpha(\tau)-1)}.
\end{equation}
\begin{prop}\label{prop:ss}
Equation (\ref{e:reeder}) holds for all semisimple root data $\Rc$. 
\end{prop}
\begin{proof} Let $G'$ be an arbitrary connected complex semisimple group with root datum $\Rc'$, and 
let $1\to C\to G\to G'\to 1$ be the universal covering of $G'$. Consider $x'=s'u'\in G'$ elliptic, 
and let $x=su\in G$ be a lifting of $x$. 
Let $A(x), A(x')\subset G_{ad}$ be the centralizers of the unramified Langlands parameters 
associated to $x$ and $x'$ as in the text above Proposition \ref{prop:ss} 
(these are finite subgroups, since $x$ and $x'$ are elliptic). 
We define a homomorphism $A(x')\to C$ by $a\to asa^{-1}s^{-1}\in C$, whose image we denote by $C_x$.
Similar to \cite[Section 3.3]{Ree4}, this give rise to an exact sequence:  
\begin{equation*}
1\to A(x)\to A(x')\to C_x\to 1
\end{equation*}
(more precisely, $A(x)$ and $A(x')$
are the images in $G_{ad}$ of Reeder's $A_{\tau,u}$ and $A_{\tau,u}^+$ respectively). 
Let $(x,\phi)\in \textup{DS}_{KL}(\Rc)$. Let $C_{x,\phi}\subset C_x$ 
be the isotropy group of $\phi$ with respect to the natural action of $C_x$ on the 
set of equivalence classes of irreducible characters of $A(x)$. Let $\mu$ denote the 
complex $2$-cocycle of $C_{x,\phi}$ associated to $\phi$, and 
$E_{x,\phi}=\mathbb{C}[C_{x,\phi},\mu]$ the corresponding twisted group algebra.
By Mackey theory  (see \cite[Section 3.3]{Ree4}) we have 
$E_{x,\phi}\simeq \textup{End}_{ \mathbb{C}[A(x')]}(\textup{Ind}_{A(x)}^{A(x')}\phi)$, and thus 
\begin{equation}\label{ex}
\textup{Ind}_{A(x)}^{A(x')}\phi=\bigoplus_{\psi\in\textup{Irr}(E_{x,\phi})}\psi\otimes \rho_\phi^\psi 
\end{equation}
with $\rho_\phi^\psi\in \textup{Irr}(A(x'))$  
as $E_{x,\phi}\otimes \mathbb{C}[A(x')]$-module.
Moreover, all irreducible characters $\rho$ of $A(x')$ 
appear as some $\rho\simeq \rho_\phi^\psi$, and $\rho_\phi^\psi\simeq \rho_{\phi'}^{\psi'}$ 
if and only if $\phi'$ is a twist of $\phi$ by an element of 
$C_x$, and $\psi'$ and $\psi$ correspond accordingly via this twist. By counting the
multiplicity of $\rho_\phi^\psi$ in the regular representation of $A(x')$ using $(\ref{ex})$ we see: 
\begin{equation}\label{dim}
\textup{dim}(\rho_\phi^\psi)=\frac{|C_{x}|}{|C_{x,\phi}|}\textup{dim}(\phi)\textup{dim}(\psi)
\end{equation}
On the other hand we consider the affine Hecke algebras $\Hc_v$ and $\Hc'_v$. Reeder \cite[Section 1.5, Lemma 3.5.2]{Ree4} showed 
$\Hc'_v=(\Hc_v)^C$, and 
\begin{equation}
\pi_{v,b_{x,\phi}}|_{\Hc'_v}=\bigoplus_{\psi\in\textup{Irr}(E_{x,\phi})}\psi\otimes \pi_{v,b_{x',\rho_\phi^\psi}}
\end{equation}
Now $\pi_{v,b_{x',\rho_\phi^\psi}}$ will arise as a summand of $\pi_{v,b_{\tilde{x},\tilde{\phi}}}|_{\Hc'_v}$ if and 
only if $(\tilde{x},\tilde{\phi})$ is a twist of $(x,\phi)$ by an element of $C$. 
The Plancherel decomposition of the trace $\tau'$ of the normalized Hecke algebra $\Hc'_v$ is obtained by 
restricting the Plancherel decomposition of the trace $\tau$ of $\Hc_v$, since  (see 
\cite[Paragraph 2.4.1]{Opd4}) we have $\tau'=\tau|_{\Hc'_v}$.
The above shows that in this restriction, the character of  
$\pi_{v,b_{x',\rho_\phi^\psi}}$ will appear with formal degree equal to the formal degree of $\pi_{v,b_{x,\phi}}$ with 
respect to $\Hc_v$, 
multiplied by the multiplicity 
\begin{equation}\label{mult}
\frac{|C|}{|C_{x,\phi}|}\textup{dim}(\psi)
\end{equation} 
Considering that clearly $m_v(\tau)=m_v(\tau')$, 
and using (\ref{e:reeder}), (\ref{ex}), (\ref{dim}), and (\ref{mult}), we see that 
$\textup{fdeg}(\pi_{v,b_{x',\rho_\phi^\psi}})=\frac {\rho_\phi^\psi(1)}{|A(x')||Z(G')|} m_v(\tau')$ as was to be proved.
\end{proof}
By \cite[Proposition 7.2]{Ree1}, $m_v(\tau)=q^{\dim \Bc_u}(q-1)^\ell |M_0^s| R(q)$, where $\ell=\dim T$, $M_0^s$ is as in \cite[Lemma 7.1]{Ree1}, and $R(q)$ is a rational function in $q$ that has the property that $R(0)=1$ and $R(1)\neq 0$. On the other hand, we know by \cite[Proposition 3.27(v)]{Opd1} that $m_v(\tau) q^{-\dim\Bc_u}$ must equal a scalar times a rational function in cyclotomic polynomials in $q$. This means that $R(q)$ is a scalar times a rational function where both the numerator and the denominator are products of cyclotomic polynomials $\Phi_n(q)$ with $n\ge 2$.  Since $R(0)=1$,  the scalar must be in fact $1$. But then $R(q)>0$ whenever $q$ is specialized to any real number greater than $-1$, which implies that $m_v(\tau)>0$ whenever $q>1.$

Notice that the same conclusion follows from our sign formula (\ref{e:m-bar}). In the equal parameter case, we have $k_s\equiv 1$, and therefore the number of $(-1)$ in (\ref{e:m-bar})  contributions equals the number of roots $\alpha(\bar c)<0$ plus the number of roots $\alpha(\bar c)<1.$ Because of integrality of the central character $\bar c$ and the fact that the roots that vanish of $\bar c$ come in pairs, it follows, that this is an even number, thus the sign of $\bar m_{W_0\bar c}$ is positive.

\begin{prop}
Suppose $\Rc$ is a semisimple root datum and let $\vb_0$ denote the specialization in the equal parameters case.  
\begin{enumerate}
\item[(a)] $\Bc_{\vb_0-m}=\{b_{x,\phi}:=\epsilon(x,\phi) h_{x,\phi}\mid h_{x,\phi}\in \Bc_{DS}^{KL}\}$, where $\epsilon(x,\phi)=\epsilon(b_{x,\phi};\vb_0)$ is the sign of the appropriate limit $\lim_{\vb\to \vb_0}m_{W_0r_b(\vb)}$ in the notation of Remark \ref{r:sign-spec}.
\item[(b)] If $\Rc$ is simply laced, then $\Bc_{gm}=\Bc_{DS}^{KL}$ and $d_{b_{x,\phi}}=\frac {\phi(1)}{|A(x)||Z(G)|}$ for all $b_{x,\phi}\in \Bc_{gm}.$
\item[(c)]Suppose $\Rc$ is a root datum of type $G_2$ or $F_4$. Then  $$\Bc_{gm}(G_2)=\{\epsilon(x,\phi) h_{x,\phi}\mid h_{x,\phi}\in \Bc_{DS}^{KL}\}$$ and  $$\Bc_{gm}(F_4)=\{\epsilon(x,\phi) h_{x,\phi}\mid h_{x,\phi}\in \Bc_{DS}^{KL}\}\cup \{b_{11}=H^\bullet(\Bc_{x_{B_4}})^{\textup{triv}}\otimes \varepsilon\},$$ with $x_{B_4}=s_{B_4}u_{(711)}$, where $Z_{G}(s_{B_4})$ is of type $B_4$ and $u_{(711)}$ is a representative of the subregular unipotent class in $B_4.$  The explicit signs $\epsilon(x,\phi)$ are given in the  sixth column in Tables \ref{t:G2} and \ref{t:F4}, while the constants $d_{b_{x,\phi}}$ are listed in the fourth column of the tables. 
\end{enumerate}
\end{prop}

\begin{proof} As in Remark \ref{r:sign-spec}, for every $b\in \Bc_{\vb_0-m}$, $\textup{ds}(b;\vb_0)=\epsilon(b;\vb_0)\textup{Ind}(b;\vb_0)$ 
we know that the sign $\epsilon(b,\vb_0)$ is given by the sign of the limit of the generic $m$-function. Applying the restriction map, we get that $\textup{Res}(\textup{ds}(b;\vb_0)=\epsilon(b;\vb_0)\textup{Ind}(b;\vb_0)$. On the other hand, as explained above $\textup{Res}(\textup{ds}(b;\vb_0)$ equals an $ h_{x,\phi}$ for some $(x,\phi)\in \textup{DS}_{KL}(\mathcal R).$ This is the claim in part (a).

Part (b) follows immediately from (a) by Proposition \ref{prop:ss} and the discussion preceding 
the present proposition: in the simply-laced case, each limit $\lim_{\vb\to \vb_0}m_{W_0r_b(\vb)}$ 
equals one $m_v(\tau)$, and they are all positive.

Part (c) is contained in subsection \ref{subsec:G2}, where Tables \ref{t:G2} and \ref{t:F4} were computed.

\end{proof}

\begin{rem} If $\Rc$ is a simply connected root datum of type $B$ or $C$, then one can relate the elements of  $\Bc_{DS}^{KL}$ to the bipartitions  in the basis $\Bc_{gm}$ from Proposition \ref{p:B-gm-C} using the combinatorial algorithms of Slooten \cite{Slo}, see also \cite{Opd4}.  If $(x,\phi)$ is a discrete Kazhdan-Lusztig parameter as above, let $(\lambda(x,\phi),\mu(x,\phi))$ be the bipartition associated by the algorithms in {\it loc. cit.} Then $$\Bc_{DS}^{KL}=\{\epsilon(\lambda(x,\phi),\mu(x,\phi))\cdot  b_{(\lambda(x,\phi),\mu(x,\phi))} \mid (x,\phi)\in \textup{DS}_{KL}(\Rc)\},$$ where $\epsilon(\lambda(x,\phi),\mu(x,\phi))$ equals the sign of $\lim_{k_s(\alpha)\to 1} \bar m_{W_0\bar c}$ from (\ref{e:m-bar}).
\end{rem}

\subsection{Pin cover of the Weyl group} Suppose $\mathcal W$ is a finite Weyl group with its reflection representation $\mathcal E.$ Fix a positive definite $\mathcal W$-invariant symmetric bilinear form on $\mathcal E.$ Define the Clifford algebra $C(\mathcal E)$ and the pin cover $p:\widetilde{\mathcal W}\to \mathcal W$ as in \cite[section 3.1]{COT}. Let $\det$ be the determinant character of $\widetilde {\mathcal W}$ acting on $\mathcal E$. As in \cite[section 4.1]{COT}, define $\widetilde {\mathcal W}'$ to be equal to $\widetilde {\mathcal W}$, when $\dim \mathcal E$ is odd, and to equal $\ker\det$ (an index two subgroup), when $\dim \mathcal E$ is even.

If $\dim \mathcal E$ is odd, then $C(\mathcal E)$ has two non-isomorphic simple complex modules, we denote them $S^+$ and $S^-$. When $\dim\mathcal E$ is even, $C(\mathcal E)$ has a unique simple complex module whose restriction to the even part $C_0(\mathcal E)$ splits into a direct sum of two non-isomorphic modules, denoted again $S^+$ and $S^-$. We fix the choice of $S^+$ (and $S^-$) in both cases once and for all.

An irreducible $\widetilde {\mathcal W}'$-representation is said to be genuine if it does not factor through $p(\widetilde {\mathcal W}')$. Two non-isomorphic irreducible $\widetilde {\mathcal W}'$-representations are said to be associate if one is the $\det$-dual of the other, when $\dim \mathcal E$ is odd, and if they both occur in the restriction to $\widetilde {\mathcal W}'$ of an irreducible $\widetilde {\mathcal W}$-module, when $\dim\mathcal E$ is even. For example $\{S^+,S^-\}$ is a pair of associate  $\widetilde {\mathcal W}'$-representations. Denote by
$\textup{Irr}^2\widetilde {\mathcal W}'$ the set of associate genuine pairs.

We apply these constructions in the case when $\mathcal R$ is a semisimple, simply-connected root datum, $\mathcal W=W_s$, with $s$ an isolated element of $T_u$ and $\mathcal E=T_s(T)\cong E.$  For the connections with elliptic theory, see \cite[section 4]{COT} for example. Given $\{\widetilde\sigma^+,\widetilde\sigma^+\}$, let $\xi_s(\widetilde\sigma^+)=\xi_s(\widetilde\sigma^-)$ be a representative of the generic central character defined by  \cite[Theorem 3.2]{COT}. 
Denote $$\textup{Irr}^2_{gm}\widetilde {W_s}'=\{\{\widetilde\sigma^+,\widetilde\sigma^-\}\in \textup{Irr}^2\widetilde W_s' \mid \xi_s(\widetilde\sigma^+)\text{ is generically residual}\}$$ and
$$\textup{Irr}^2_{KL}\widetilde {W_s}'=\{\{\widetilde\sigma^+,\widetilde\sigma^-\}\in \textup{Irr}^2_{gm}\widetilde W_s' \mid \xi_s(\widetilde\sigma^+)(1)\text{ is residual}\},$$
where $\xi_s(\widetilde\sigma^+)(1)$ denotes the specialization at the equal parameter case $k_s=1.$ 

Write $\Bc_{gm}=\sqcup_{s} \textup{Ind}_s(\Bc_{gm,s})$ for the canonically positive basis, where $\Bc_{gm,s}$ is a certain orthonormal subset of  $\overline R_{\mathbb Z}(W_s)$. By \cite[Theorem 4.2]{COT}, for each $b_s\in\Bc_{gm,s}$ there exists a pair $(\widetilde\sigma_{b_s}^+,\widetilde\sigma_{b_s}^-)$ of associate genuine irreducible $\widetilde W_s'$-representations such that
$$b_s=\Delta({\widetilde\sigma_{b_s}^+,\widetilde\sigma_{b_s}^-}):=\frac{\widetilde\sigma_{b_s}^+-\widetilde\sigma_{b_s}^-}{S^+-S^-}\text{ in }\overline R_{\mathbb Z}(W_s).$$
Moreover, as explained in subsection \ref{subsec:dirac}, $\xi_s(\widetilde\sigma_{b_s}^+)=r_{b_s}$, in particular, $\{\widetilde\sigma_{b_s}^+,\widetilde\sigma_{b_s}^-\}\in\textup{Irr}^2_{gm}\widetilde {W_s}'.$ By the same results in {\it loc. cit.}, the assignment $b_s\mapsto \{\widetilde\sigma_{b_s}^+,\widetilde\sigma_{b_s}^-\}$ is injective. Thus if we define $\underline{\textup{Irr}}^2_{gm}\widetilde {W_s}'=\{(\widetilde\sigma_{b_s}^+,\widetilde\sigma_{b_s}^-)\mid b_s\in\Bc_{gm,s}\}$, then the map $\Delta$ defines a bijection $$\Delta:\underline{\textup{Irr}}^2_{gm}\widetilde {W_s}'\to \Bc_{gm,s}.$$
In fact, $\underline{\textup{Irr}}^2_{gm}\widetilde {W_s}'$ is just $\textup{Irr}^2_{gm}\widetilde {W_s}'$ with a particular choice of $\widetilde\sigma^+$ versus $\widetilde\sigma^-$.

\medskip

We can explain this choice independently in the case of $\textup{Irr}^2_{KL}\widetilde {W_s}'$. Since $\xi_s(\widetilde\sigma^+)(1)$ is a residual point, we know by \cite{Opd1} that $\xi_s(\widetilde\sigma^+)(1)$ satisfies the same combinatorial condition as the Bala-Carter condition for half of the middle element of a distinguished Lie triple in $Z_G(s)$.  Thus, the pair $\{\widetilde\sigma^+,\widetilde\sigma^-\} \in \textup{Irr}^2_{KL}\widetilde {W_s}'$ determines the (conjugacy class of an) element $u$, and therefore the elliptic element $x=su$ in the Kazhdan-Lusztig parameterization. To formalize this, denote by $u(\{\widetilde\sigma^+,\widetilde\sigma^-\})$ a representative of the unipotent class attached in this way. 

Let $\Omega_{\widetilde W_s}$ be the Casimir element of $\widetilde W_s$, e.g., \cite[(3.2.5)]{COT}. The scalar $\widetilde\sigma^+(\Omega_{\widetilde W_s})$ by which $\Omega_{\widetilde W_s}$ acts in $\widetilde \sigma^+$ (which does not depend on the choice of $\widetilde\sigma^+$ versus $\widetilde\sigma^-$) equals the squared norm of $\xi_s(\widetilde\sigma^+)$.

It remains to discuss how to associate a representation $\phi$ of the component group $A(x)$. This will also lead to the desired canonical choice for $\widetilde\sigma^\pm$. Let $x=su$ be an elliptic element of $G$. Recall that this means that $u$ is distinguished in $Z_G(s)$. Let $X_s(u,\phi)$ and $\sigma_s(u,\phi)$ denote the standard and irreducible $W_s$-representations afforded via Springer theory by $H^{\bullet}(\Bc^s_u)^\phi\otimes\varepsilon$ and $H^{2d^s_u}(\Bc^s_u)^\phi\otimes\varepsilon$ respectively, where $\Bc^s_u$ is the Springer fiber of $u$ in $Z_G(s)$ and $d^s_u=\dim \Bc^s_u$. As above, there exist associate irreducible representations $\widetilde\sigma_s(u,\phi)^\pm$ such that
$$X_s(u,\phi)\otimes (S^+-S^-)=\widetilde\sigma_s(u,\phi)^+-\widetilde\sigma_s(u,\phi)^-.$$
As noticed in \cite{CT,CH2}, if we write $X_s(u,\phi)=\sigma_s(u,\phi)+\sum_{u'>u}a_{(u',\phi')}X_s(u',\phi')$ (by the results of Borho-Macpherson), then one can deduce that $\widetilde\sigma_s(u,\phi)^\pm$ can only appear in $\sigma_s(u,\phi)\otimes S^\pm$ and in no other $X_s(u',\phi')\otimes S^\pm$. Moreover, by the Dirac theory, the scalar $\widetilde{\sigma}_s(u,\phi)^\pm(\Omega_{\widetilde W_s})$ equals the squared norm of half a middle element for a Lie triple of $u$, and this is the {\it minimal} such scalar that may appear for an irreducible constituent of $\sigma_s(u,\phi)\otimes S^\pm$. By \cite{Ci} (case-by-case) or \cite{CH2} (uniformly), $\widetilde\sigma_s(u,\phi)^+$ appears with multiplicity one in $\sigma_s(u,\phi)\otimes S^+$ (and similarly for the $-$ case). 

The conclusion is that $\widetilde\sigma_s(u,\phi)^+$ is characterized by the property that it is the unique irreducible constituent of  $\sigma_s(u,\phi)\otimes S^+$ for which the scalar by which $\Omega_{\widetilde W_s}$ acts in it is minimal among all the constituents of the tensor product. Moreover, $\widetilde\sigma_s(u,\phi)^+$ does not appear as such a ``minimal representation" in any other tensor product of this form. This gives a one-to-one correspondence between the sets $\{\sigma_s(u,\phi)\}$ and $\{\widetilde\sigma_s(u,\phi)^+\}.$ Given a $\widetilde \sigma^+$ in the latter set, denote by $\phi(\widetilde\sigma^+)$ the local system $\phi$ for the Springer representation in the former set.  In particular, we have a preferred representation $\widetilde\sigma^+$ in any pair belonging to $\textup{Irr}^2_{KL}\widetilde {W_s}'$. Thus we get a set $\underline{\textup{Irr}}^2_{KL}\widetilde {W_s}'$ of ordered pairs.
The discussion can be summarized in the following 

\begin{prop}\label{p:comparison}
Suppose $\Rc$ is a simply-connected semisimple root datum.  Then, with the canonical choice of representations $\widetilde\sigma^+$ as above, $$\bigsqcup_{s}{\textup{Ind}}_s\circ \Delta (\underline{\textup{Irr}}^2_{KL}\widetilde {W_s}')=\Bc^{KL}_{DS}.$$ Given a pair $(\widetilde\sigma^+,\widetilde\sigma^-)\in  \underline{\textup{Irr}}^2_{KL}(\widetilde {W_s}')$, the Kazhdan-Lusztig parameter $(x=su,\phi)$ is given by $u=u(\{\widetilde\sigma^+,\widetilde\sigma^-\})$ and $\phi=\phi(\widetilde\sigma^+)$ with the notation as in the previous paragraph. 
\end{prop} 

\begin{rem}
If, in addition $\Rc$ is simply-laced, this also implies that $\Bc_{gm}=\bigsqcup_{s}\textup{Ind}_s\circ \Delta (\underline{\textup{Irr}}^2_{KL}\widetilde {W_s}').$ When $\Rc$ is not simply-laced, the relation between  $\Bc_{\vb_0-m}$ and $\Bc^{KL}_{DS}$ (and therefore $\bigsqcup_{s}\textup{Ind}_s\circ \Delta (\underline{\textup{Irr}}^2_{KL}\widetilde {W_s}')$) was given in section \ref{s:KL}.
\end{rem}



\begin{rem}
The case of root data of type $E$ is particularly interesting.
Suppose $\Rc$ is a simply-connected, type $E$, root datum. Let $\{\widetilde\sigma^+,\widetilde\sigma^-\}\in  \underline{\textup{Irr}}^2_{gm}(\widetilde {W_s}')$ be given, with the notation as above. Then the unipotent class of $u(\{\widetilde\sigma^+,\widetilde\sigma^-\})$ from Proposition \ref{p:comparison} can be effectively determined as follows. The key, empirical, observation is that for root systems of type $E$, the squared norm of the middle element of a quasidistinguished Lie triple uniquely identifies the unipotent class
Hence the scalar by which $\Omega_{\widetilde W_s}$ acts in $\widetilde \sigma^+$ already identifies uniquely the unipotent class of $u(\{\widetilde\sigma^+,\widetilde\sigma^-\})$!
\end{rem}

\subsection{Rationality of characters of discrete series} In this subsection, we discuss the character value ring for the families of generic discrete series considered before. In light of Corollary \ref{cor:cont}, we know that for every generic residual central character $W_0r$ and for every $\vb\in\Qc^{reg}_{W_0r}$, the central functional given by the ``stable combination" of generic characters (\ref{eq:dschar}) takes values in $Q_\Lambda$, the quotient field of $\Lambda$, i.e.:
\begin{equation}\label{eq:rational-sum}
\sum_{\chi\in \textup{DS}_{W_0r(\vb)}}c_{\chi,C}\chi_\vb(h)\in Q_\Lambda,\text{ for all }h\in \Hc_\Lambda.
\end{equation}

\begin{prop}Suppose the root datum $\Rc$ of the generic affine Hecke algebra $\Hc_\Lambda$ is simple and not simply-laced. Let $W_0r$ be a generic residual central character and let $\vb\in\Qc^{reg}_{W_0r}$ be given. Then for every discrete series character $\chi\in \textup{DS}_{W_0r(\vb)}$, we have $\chi_\vb(h)\in Q_\Lambda,$ for all $h\in \Hc_\Lambda.$
\end{prop}

\begin{proof}
By the classification of generic discrete series families \cite{Opd1} and \cite{OpdSol2}, we know that, when $\Rc$ is not simply-laced, the cardinality of $\textup{DS}_{W_0r(\vb)}$ is always $1$, except if $\Rc=F_4$ and $W_0r$ is the central character $[0,k_1,0,k_2-k_1]$ from Table \ref{t:F4}. Therefore, with this exception, the claim follows from (\ref{eq:rational-sum}) immediately. 

Let's consider now this exceptional residual central character in $F_4$. In the same coordinates as in Table \ref{t:F4}, we have $r=v^{2[0,k_1,0,k_2-k_1]}$ and $\Qc^{reg}_{W_0r}=\{k_1,k_2\in\mathbb R\mid k_1k_2\neq 0\}.$ There are always two families of discrete series $\pi_{\vb,I}=\ep(b_8;\vb)\textup{Ind}_D(b_8;\vb)$ and $\pi_{\vb,II}=\ep(b_9;\vb)\textup{Ind}_D(b_9;\vb)$ for $\vb\in \Qc^{reg}_{W_0r}$, where $b_8,b_9$ are as in Table \ref{t:F4}. From (\ref{eq:rational-sum}), we know that $\pi_{\vb,I}+\pi_{\vb,II}$ takes values in $Q_\Lambda$, so it suffices to show that $\pi_{\vb,II}$ has the same property. Assume first that $k_1>0$ and $k_2>0$. Then $\pi_{\vb,II}$ is a $10$-dimensional module that has the property that it is  $\Ac$-semisimple, i.e., each generalized weight space under the action of Bernstein's abelian subalgebra $\Ac$ is one-dimensional. This claim follows from \cite[page 80]{Ree1}, where the weight diagram of this module is given under the label $[A_1E_7(a_5),-21]$.  Let $\textup{Wt}(\pi_{\vb,II})$ denote the set of weights. As it is well known, we may choose a basis of the module consisting of weight vectors $\mathcal y_\lambda$, $\lambda\in \textup{Wt}(\pi_{\vb,II})$ such that for every $\bs\in S_0$, the action of $N_\bs$ is given by
$$N_\bs\cdot \mathcal y_\lambda=\begin{cases} \frac{v(\bs)-v(\bs)^{-1}}{1-\theta_{-\alpha}(\lambda)}\mathcal y_\lambda,&\text{ if }\bs(\lambda)\notin \textup{Wt}(\pi_{\vb,II}),\\
\frac{v(\bs)-v(\bs)^{-1}}{1-\theta_{-\alpha}(\lambda)}\mathcal y_\lambda+\left(v(\bs)^{-1}+\frac{v(\bs)-v(\bs)^{-1}}{1-\theta_{-\alpha}(\lambda)}\right)\mathcal y_{\bs(\lambda)}&\text{ if }\bs(\lambda)\in \textup{Wt}(\pi_{\vb,II}),\\
\end{cases}
$$ 
where $\bs=\bs_\alpha.$ This means that both types of generators, $\theta_x$ and $N_\bs$ act in this basis via matrices with entries in $Q_\Lambda$, and the claim follows.

To treat the other three quadrants in $k_1,k_2$, notice that the affine Hecke algebra of type $F_4$ has three involutions of the form $N_\bs\mapsto -N_\bs^{-1}$, where $\bs$ ranges over the long simple reflections, the short simple reflections, and all the simple reflections, respectively. (The last one is the Iwahori-Matsumoto involution.) Applying these involutions to the modules $\pi_{\vb,I}$ and $\pi_{\vb,II}$ from the $k_1>0,k_2>0$ case, we obtain the discrete series in the other three cases.

\end{proof}

\newpage
\newgeometry{left=3cm,bottom=0.1cm}

\begin{landscape}

\begin{table}[h]
\caption{$G_2$\label{t:G2}}
\begin{tabular}{|c|c|c|c|c|c|c|c|c|c|}
\hline
$b$ &$s$ & $W_0 c$ &$d_b$ &$G_2$   &$\epsilon$ &$E_6\subset E_8$ &$\epsilon$ &${}^3E_6$ &$\epsilon$\\

\hline
$b_1$ &$1$ &$ {[k_1,k_2]}$ &$1$ &$[G_2,1]$ &$1$ &$[A_2E_6,\theta]$ &$1$ &$E_6$ &$1$\\
\hline
$b_2$ &$1$ &$ {[k_1,-k_1+k_2]}$ &$1$ &$[G_2(a_1),(21)]$ &$-1$ &$[A_2E_6(a_1),\theta]$ &$-1$ &$E_6(a_1)$ &$-1$\\
\hline
$b_3$ &$1$ &$ {[k_1, \frac 12 (-k_1+k_2)]}$ &$1/2$ &$[G_2(a_1),(3)]$ &$1$ &$[A_2E_6(a_3),\theta]$ &$1$ &$E_6(a_3)$ &$1$\\
\hline
$b_4$ &$2A_1$ &$  {[-\frac 12 {k_1}-\frac 32 k_2,k_2]}$ &$1/2$ &$[2A_1,1]$ &$1$ &$[A_1A_2A_5,\theta]$ &$1$ &$A_1A_5$ &$1$\\
\hline
$b_5$ &$A_2$ &$  {[k_1,-k_1]}$ &$1/3$ &$[A_2,1]$ &$1$ &$[A_8,\theta]$ &$1$ &$A_2^3$ &$1$\\
\hline
\end{tabular}
\end{table}


\begin{table}[h]
\caption{$F_4$\label{t:F4}}
\begin{tabular}{|c|c|c|c|c|c|c|c|c|c|}
\hline
$b$ &$s$ & $W_0 c$ &$d_b$  &$F_4$  &$\epsilon$ &$D_4\subset E_8$ &$\epsilon$ &${}^2E_7$ &$\epsilon$\\

\hline
$b_1$ &$1$ &$ {[k_1,k_1,k_2,k_2]}$ &$1$ &$[F_4,1]$ &$1$ &$[A_1E_7,-]$&$1$ &$E_7$ &$1$\\
\hline
$b_2$ &$1$  &$ {[k_1,k_1,k_2-k_1,k_2]}$ &$1$ &$[F_4(a_1),-]$ &$-1$ &$[A_1E_7(a_4),--]$ &$-1$ &$E_7(a_1)$ &$-1$\\
\hline
$b_3$ &$1$  &$ {[k_1,k_1,k_2-k_1,k_1]}$ &$1$ &$[F_4(a_1),+]$ &$1$ &$[A_1E_7(a_2),-]$ &$1$ &$E_7(a_2)$ &$1$\\
\hline
$b_4$  &$1$ &$ {[k_1,k_1,k_2-2k_1,k_2]}$ &$1$ &$[F_4(a_3),(211)]$ &$1$ &$[A_1E_7(a_3),+-]$ &$1$ &$E_7(a_3)$ &$1$\\
\hline
$b_5$  &$1$ &$ {[k_1,k_1,k_2-2k_1,2k_1]}$ &$1$ &$[F_4(a_2),+]$ &$1$ &$[A_1E_7(a_3),-+]$ &$-1$ &$E_7(a_3)$ &$-1$\\
\hline
$b_6$  &$1$ &$ {[k_1,k_1,k_2-2k_1,k_1]}$ &$1$ &$[F_4(a_3),(31)]$ &$-1$ &$[A_1E_7(a_4),+-]$ &$1$ &$E_7(a_4)$ &$1$\\
\hline
$b_7$  &$1$ &$ {[k_1,k_1,k_2-2k_1,-2k_2]}$ &$1$ &$[F_4(a_2),-]$ &$-1$ &$[A_1E_7(a_1),-]$ &$-1$ &$E_7(a_4)$ &$-1$\\
\hline
$b_8$  &$1$ &$ {[0,k_1,0,k_2-k_1]}$ &$1/6$ &$[F_4(a_3),(4)]$ &$1$ &$[A_1E_7(a_5),-3]$ &$1$ &$E_7(a_5)$ &$1$\\
\hline
$b_9$  &$1$ &$ {[0,k_1,0,k_2-k_1]}$ &$1/3$ &$[F_4(a_3),(22)]$ &$1$ &$[A_1E_7(a_5),-21]$ &$1$ &$E_7(a_5)$ &$1$\\
\hline
$b_{10}$  &$B_4$ &$ {[\frac {k_1}2,k_1,k_2,-3k_1-2k_2]}$ &$1/2$  &$[B_4,+]$ &$1$ &$[D_8,-]$&$1$ &$A_1D_6$ &$1$\\
\hline
$b_{11}$  &$B_4$ &$ {[2k_1,-k_1,k_2,-k_1-2k_2]}$ &$1/2$ &non-ds & &$[D_8(5,11),-]$ &$1$ &$A_1D_6(3,9)$ &$-1$\\
\hline
$b_{12}$  &$B_4$ &$  {[0,k_1,-k_1+k_2,-2k_2]}$ &$1/2$  &$[B_4(531),\epsilon'']$ &$-1$ &$[D_8(1,3,5,7),r]$ &$1$ &$A_1D_6(5,7)$ &$-1$\\
\hline
$b_{13}$  &$B_4$ &$  {[k_1,k_1,-2k_1+k_2,k_1-2k_2]}$ &$1/2$  &$[B_4(531),1]$ &$1$ &$[D_8(7,9),-]$ &$-1$ &non-ds &\\
\hline
$b_{14}$  &$B_4$ &$  {[k_1,k_1,-3k_1+k_2,3k_1-2k_2]}$ &$1/2$ &$[B_4(531),\epsilon']$ &$-1$ &$[D_8(3,13),-]$ &$-1$ &non-ds &\\
\hline
$b_{15}$&$C_3A_1$  &$  {[-2k_1-3k_2,k_1,k_2,k_2]}$ &$1/2$ &$[C_3\times A_1,+]$  &$1$ &$[A_3D_5,-1]$ &$1$ &$A_1D_6$ &$1$\\
\hline
$b_{16}$ &$C_3A_1$ &$  {[-2k_1,k_1,-k_2,2k_2]}$ &$1/2$ &$[C_3(42)\times A_1,++]$ &$1$ &$[A_3D_5(3,7),-1]$ &$-1$ &$A_1D_6(5,7)$ &$1$\\
\hline
$b_{17}$&$C_3A_1$  &$  {[-2k_1+3k_2,k_1,-k_2,-k_2]}$ &$1/2$ &$[C_3(42)\times A_1,+-]$\ &$-1$ &non-ds & &$A_1D_6(3,9)$ &$-1$\\
\hline
$b_{18}$ &$2A_2$ &$ {[k_1,-k_1-2k_2,k_2,k_2]}$ &$1/3$ &$[2A_2,1]$ &$1$ &$[A_1A_2A_5,-1]$ &$1$ &$A_2A_5$ &$1$\\
\hline
$b_{19}$ &$A_3A_1$ &$ {[k_1,k_1,-\frac {3k_1}2 -\frac  {k_2}2,k_2]}$ &$1/4$ &$[A_1A_3,1]$ &$1$ &$[A_1A_7,-]$ &$1$ &$A_1A_3^2$ &$1$\\
\hline
\end{tabular}
\end{table}
\end{landscape}

\newgeometry{left=3cm,bottom=0.1cm}

\end{document}